\documentclass[a4paper,12pt,oneside,openright,reqno]{amsart}
\usepackage[colorlinks,citecolor=green,linkcolor=red]{hyperref}
\usepackage[english]{babel}
\usepackage{color}
\usepackage{mathrsfs}
\usepackage{amsmath}
\usepackage{mathtools}
\usepackage{amsfonts}
\usepackage{amssymb}
\usepackage{physics}
\usepackage{xfrac}
\usepackage{a4wide}
\usepackage{esint}
\usepackage{nicefrac}
\usepackage{yfonts}
\usepackage{lipsum}
\usepackage{amsthm}
\usepackage{color}
\usepackage{mathrsfs}
\usepackage{amsmath}
\usepackage{mathtools}
\usepackage{amsfonts}
\usepackage{amssymb}
\usepackage{bm}
\usepackage{physics}
\usepackage{enumitem}
\usepackage{tikz-cd}
\usepackage{a4wide}
\usepackage{cleveref}

\newtheorem{cnj}{Conjecture}

\newtheorem{thm}{Theorem}
\newtheorem*{thm*}{Theorem}
\newtheorem{lem}[thm]{Lemma}
\newtheorem{prop}[thm]{Proposition}

\theoremstyle{definition}
\newtheorem{defn}[thm]{Definition}

\theoremstyle{remark}
\newtheorem{rem}[thm]{Remark}

\newcommand{\supp}{{\mathrm {supp\,}}}

\newcommand{\mres}{\mathbin{\vrule height 1.6ex depth 0pt width 0.13ex\vrule height 0.13ex depth 0pt width 1.3ex}}

\newcommand{\defeq}{\mathrel{\mathop:}=}
\newcommand{\RR}{\mathbb{R}}
\newcommand{\LL}{\mathcal{L}}

\newcommand{\EE}{\mathcal{E}}
\newcommand{\DD}{\mathcal{D}}
\newcommand{\HH}{\mathcal{H}}
\newcommand{\BB}{\mathcal{B}}

\newcommand{\MM}{\mathcal{M}}
\newcommand{\SSS}{\mathcal{S}}

\newcommand{\DIFF}{{\mathrm {D}}}

\newcommand{\BV}{{\mathrm {BV}}}

\newcommand{\rmc}{{\mathrm {c}}}

\newcommand{\rmloc}{{\mathrm {loc}}}
\def\Xint#1{\mathchoice
	{\XXint\displaystyle\textstyle{#1}}%
	{\XXint\textstyle\scriptstyle{#1}}%
	{\XXint\scriptstyle\scriptscriptstyle{#1}}%
	{\XXint\scriptscriptstyle\scriptscriptstyle{#1}}%
	\!\int}
\def\XXint#1#2#3{{\setbox0=\hbox{$#1{#2#3}{\int}$}
		\vcenter{\hbox{$#2#3$}}\kern-.5\wd0}}

\def\dashint{\Xint-}
\newcommand{\fr}{\penalty-20\null\hfill$\blacksquare$}     
\let\epsilon\varepsilon

\DeclareMathOperator*{\argmin}{arg\,min}

\let\epsilon\varepsilon

\title{Linear Inverse Problems with \\Hessian-Schatten Total Variation} 

\author{Luigi Ambrosio*}
\thanks{*Scuola Normale Superiore di Pisa,  Pisa,  Italy (luigi.ambrosio@sns.it,  camillo.brena@sns.it)}
\author{Shayan Aziznejad$^\dag$}
\thanks{$^\dag$Biomedical Imaging Group,  E\'cole Polytechnique Fe\'de\'rale de Lausanne,  Lausanne, Switzerland (sh.aziznejad@gmail.com,  michael.unser@epfl.ch)}
\author{Camillo Brena*}
\author{Michael Unser$^\dag$}

\begin{document}
\maketitle
\begin{abstract}
  In this paper, we characterize the class of extremal points of the unit ball of the Hessian-Schatten total variation (HTV) functional. The underlying motivation for our work stems from a general representer theorem  that characterizes the solution set of regularized linear inverse problems in terms of the  extremal points of the regularization ball. Our analysis is mainly based on studying the class of continuous and piecewise linear (CPWL) functions. In particular,  we show that in dimension $d=2$, CPWL functions are dense in the unit ball of the HTV functional. Moreover, we prove that a CPWL function is extremal if and only if its Hessian is minimally supported. For the converse, we prove that the density result (which we have only proven for dimension $d=2$) implies that the closure of the CPWL extreme points contains all extremal points. 
\end{abstract}
\tableofcontents
\section{Introduction}
Broadly speaking, the goal of  an inverse problem is to reconstruct an unknown signal of interest from a collection of (possibly noisy)   observations. Linear inverse problems, in particular,  are prevalent in various areas of signal processing, such as denoising, impainting, and image reconstruction. They are defined via the specification of  three principal components: (i) a hypothesis space $\mathcal{F}$ from which we aim to reconstruct the unknown signal $f^*\in\mathcal{F}$; (ii)  a linear forward operator $\boldsymbol{\nu}:\mathcal{F}\rightarrow \mathbb{R}^M$ that models the data acquisition process; and, (iii) the observed data that is stored in an array $\boldsymbol{y}\in\mathbb{R}^M$ with the implicit assumption that $\boldsymbol{y}\approx \boldsymbol{\nu}(f^*)$. The task is then to (approximately) reconstruct the unknown signal $f^*$ from the observed data $\boldsymbol{y}$.  From a variational perspective, the problem  can be formulated as a minimization of the form 
\begin{equation}\label{eq:inv_pb}
    f^* \in \argmin_{f\in\mathcal{F}} \left( E\left(\boldsymbol{\nu}(f),\boldsymbol{y}\right) + \lambda \mathcal{R}(f)\right),
\end{equation}
where $E:\mathbb{R}^M\times\mathbb{R}^M\rightarrow\mathbb{R}$ is a convex loss function that measures the data discrepancy, $\mathcal{R}:\mathcal{F}\rightarrow\mathbb{R}$ is the regularization functional that enforces prior knowledge on the reconstructed signal, and $\lambda>0$ is a tunable parameter that adjusts the two terms. 

The use of regularization for solving inverse problems dates back to the 1960s, when Tikhonov proposed a quadratic ($\ell_2$-type) functional for  solving finite-dimensional problems \cite{tikhonov1963solution}. More recently, Tikhonov regularization has been outperformed by $\ell_1$-type functionals in various settings \cite{tibshirani1996regression,donoho2003optimally}.  This is largely due to the sparsity-promoting effect of the latter, in the sense that the solution of an  $\ell_1$-regularized  inverse problem can be typically written as the linear combination of a few predefined elements, known as atoms  \cite{donoho2006most,bruckstein2009sparse}. Sparsity is a pivotal concept in modern signal processing and constitutes the core of many celebrated methods. The most notable example is the framework of compressed sensing \cite{candes2006robust,donoho2006compressed,eldar2012compressed}, which has brought lots of attention in the past decades. 

In general, regularization enhances the stability of the problem and alleviates its inherent ill-posedness, especially when the hypothesis space is much larger than  $M$. While this can happen in the discrete setting ({\it e.g.} when $\mathcal{F}=\mathbb{R}^d$ with $d\gg M$), it is inevitable in the continuum where $\mathcal{F}$ is an infinite-dimensional space of functions.  Since naturally occurring signals and images are usually indexed over the whole continuum, studying continuous-domain problems is, therefore, undeniably important. It thus comes with no surprise to see the rich literature on this class of optimization problems. Among the classical examples are the smoothing splines for interpolation \cite{schoenberg1988spline,reinsch1967smoothing} and the celebrated framework of learning over reproducing kernel Hilbert spaces \cite{wahba1990spline,scholkopf2001generalized}. Remarkably, the latter laid the foundation of numerous kernel-based machine learning schemes such as support-vector machines \cite{evgeniou2000regularization}. The key theoretical result of these frameworks is a ``representer theorem'' that provides a parametric form for their optimal solutions.   While these examples formulate optimization problems over Hilbert spaces, the representer theorem has been recently extended to cover generic convex optimization problems over Banach spaces \cite{Boyer2018,bredies2020sparsity,Unser2020b,unser2021convex}. In simple terms, these abstract results characterize the solution set of \eqref{eq:inv_pb} in terms of the extreme points of the unit ball of the regularization functional $B_{\mathcal{R}}=\{f\in\mathcal{F}: \mathcal{R}(f)\leq 1\}$. Hence, the original problem can be translated in finding the extreme points of the unit ball $B_{\mathcal{R}}$.

In parallel, Osher-Rudin-Fatemi's total-variation  has been systematically explored in the context of image restoration and denoising \cite{rudin1992nonlinear,chambolle2004algorithm,getreuer2012}. The total-variation of a differentiable function $f:\Omega \rightarrow \mathbb{R}$  can be computed as 
\begin{equation}\label{Eq:TV}
    {\rm TV}(f) = \int_{\Omega} \|\boldsymbol{\nabla} f (\boldsymbol{x})\|_{\ell_2} {\rm d}\boldsymbol{x}.
\end{equation}
The notion can be extended to cover non-differentiable functions using the   theory  of functions with bounded variation    \cite{AFP00,cohen2003harmonic}. In this case, the  representer theorem states that the solution can be written as the linear combination of some indicator functions \cite{bredies2020sparsity}. This adequately explains the so called ``stair-case effect'' of TV regularization. Subsequently, higher-order generalizations of TV regularization   have been proposed by Bredies {\it et al.}  \cite{bredies2010total,bredies2014regularization,bredies2020higher}. Particularly, the second-order TV has been used in various applications \cite{hinterberger2006variational,bergounioux2010second,knoll2011second}. By analogy with \eqref{Eq:TV}, the second-order TV is defined over the space of functions with bounded Hessian  \cite{Deme}. In particular, it can be computed for  twice-differentiable functions $f:\Omega \rightarrow \mathbb{R}$ as 
\begin{equation}\label{Eq:TV2}
    {\rm TV}^{(2)}(f) = \int_{\Omega} \|\nabla^2 f(\boldsymbol{x})\|_F {\rm d}\boldsymbol{x},
\end{equation}
where $\|{\bf \,\cdot\,}\|_F$ denotes the Frobenius norm of a matrix. Lefkimiatis {\it et al.} generalized the notion by replacing the Frobenius norm  with any Schatten-$p$ norm for $p\in [1,+\infty]$ \cite{lefki2013HS,lefki2013Poisson}. While this had been only defined for twice-differentiable functions, it has been recently extended to the  space of functions with bounded Hessian  \cite{Aziznejad2021HSTV}. The extended seminorm---the Hessian-Schatten total variation (HTV)---has also been used for learning continuous and piecewise linear (CPWL) mappings \cite{campos2021HTV,pourya2022delaunay}. The motivation and importance of the latter stems from the following observations:
\begin{enumerate}
    \item The CPWL family plays a significant role in deep learning. Indeed, it is known that the input-output mapping of any deep neural networks (DNN) with  rectified linear unit (ReLU) activation functions is a  CPWL function \cite{Montufar2014}. Conversely, any  CPWL mapping can be exactly represented by a  DNN with ReLU activation functions  \cite{arora2016understanding}.  These results provide a one-to-one correspondence between the CPWL family and the input-output mappings of commonly used DNNs. 
    \item For one-dimensional problems ({\it i.e.}, when $\Omega\subseteq\mathbb{R}$), the HTV seminorm coincides with the second-order TV. Remarkably, the representer theorem in this case states that the optimal solution can be achieved by a linear spline; that is,  a univariate CPWL function.  The latter suggests  the use of  ${\rm TV}^{(2)}$ regularization for learning univariate functions  \cite{savarese2019how,unser2019deepSpline,Aziznejad2020DeepLipschitz,Bohra020deepspline,debarre2021sparsest,Aziznejad2021RobustSparse}.
    \item It is known from the literature on low-rank matrix recovery that the Schatten-1 norm (also known as the nuclear norm) promotes low rank matrices \cite{davenport2016overview}. Hence, by using the ${\rm HTV}$ seminorm with $p=1$, one expects to obtain a mapping whose Hessian has low rank at most points, with the extreme case being the CPWL family whose Hessian is zero almost everywhere. 
\end{enumerate}
The aim of this paper is to identify the solution set of linear inverse problems with HTV regularization. Motivated by   recent general representer theorems (see, \cite{Boyer2018,unser2021convex}, we focus on the characterization of the extreme points of the unit ball of the HTV functional. After recalling some preliminary concepts (Section \ref{sec:prelim}), we study  the HTV seminorm and its associated native space from a mathematical perspective (Section \ref{sec:HTV}). Next, we prove our main theoretical result on density of CPWL functions in the unit ball of the HTV seminorm (Theorem \ref{mainthm}) in Section \ref{sec:CPWL}.  Finally, we invoke a   variant of the Krein-Milman theorem to characterize the extreme points of the unit ball of the HTV seminorm (Section \ref{sec:extreme}).
 
\section{Preliminaries}\label{sec:prelim}
Throughout the paper, we shall use   fairly standard notations  for various objects, such as function spaces and sets.  For example,  $\mathcal{L}^n$ and $\mathcal{H}^k$ denote the Lebesgue and $k$-dimensional Hausdorff  measures on $\mathbb{R}^n$,  respectively.   Below, we  recall some of the concepts that are foundational for this paper.
\subsection{Schatten norms}
\begin{defn}[Schatten norm]
	Let $p\in[1,+\infty]$. If $M\in\RR^{n\times n}$ and $s_1(M),\dots, s_n(M)\ge 0$ denote the singular values of $M$ (counted with their multiplicity), we define the Schatten $p$-norm of $M$  
	by $$|M|_p\defeq\Vert(s_1(M),\dots,s_n(M))\Vert_{\ell^p}.$$
\end{defn}
We recall that the scalar product between $M,N\in\RR^{n\times n}$ is defined by $$M\,\cdot\, N\defeq \tr(M^t N)=\sum_{i,j=1,\dots,n} M_{i,j}N_{i,j}$$ and induces the Hilbert-Schmidt norm.  Next,  we enumerate several properties of the Schatten norms that shall be used throughout the paper. We refer to standard books on matrix analysis (such as \cite{bhatia2013matrix}) for the proof of these results. 

\begin{prop}
	The family of Schatten norms satisfies the following properties.
	\begin{enumerate}
		\item If $M\in\RR^{n\times n}$ is symmetric, then its singular values $s_1(M),\dots,s_n(M)$ are equal to $|\lambda_1(M)|,\dots,|\lambda_n(M)|$, where $\lambda_1(M),\dots,\lambda_n(M)$ denote the eigenvalues of $M$ (counted with their multiplicity). Hence $|M|_p=\Vert(\lambda_1(M),\dots,\lambda_n(M))\Vert_{\ell^p}$.
		\item If $M\in\RR^{n\times n}$ and $N\in O(\RR^n)$, then $|M N|_p=|N M|_p=|M|_p$.
		\item If $M,N\in\RR^{n\times n}$, then $|M N|_p\le |M|_p |N|_p$.
		\item If $M\in\RR^{n\times n}$, then $|M|_p=\sup_N M\,\cdot\, N$, where the supremum is taken among all $N\in\RR^{n\times n}$ with $|N|_{p^*}\le 1$, for $p^*$ the conjugate exponent of $p$.
		\item If $M$ has rank $1$, then $|M|_p$ coincides with the Hilbert-Schmidt norm of $M$ for every $p\in [1,+\infty]$. 
		\item If $p\in(1,+\infty)$, then the Schatten $p$-norm is strictly convex \cite[Corollary 1]{aziznejad2021duality}.
		\item If $M\in\RR^{n\times n}$, then $|M|_p\le C|M|_q$, where  $C=C(n,p,q)$ depends only on $n$, $p$ and $q$. 
	\end{enumerate}
\end{prop}
\begin{defn}[$L^r$-Schatten $p$-norm]
	Let $p,r\in[1,+\infty]$ and let $M\in (L^r(\RR^n))^{n\times n}$. We define the $L^r$-Schatten $p$-norm of $M$ as
	$$
	\Vert M\Vert_{p,r}\defeq \Vert |M|_p\Vert_{L^r(\RR^n)}.
	$$
	An analogous definition can be given when the reference measure for the $L^r$ space is not the Lebesgue measure.
\end{defn}

\subsection{Poincaré inequalities}
We recall that for a Borel set $A\subseteq \mathbb{R}^n$ with $\mathcal{L}^n(A)>0$  and $f\in L^1(A)$,  then 
$$\dashint_A f \dd\mathcal{L}^n \defeq \frac{1}{\mathcal{L}^n(A)} \int_A f \dd\mathcal{L}^n. $$
\begin{defn}\label{defnPoin}
    Let $A\subseteq\RR^n$ be an open domain. We say that $A$ supports Poincaré inequalities if for every $q\in[1,n)$ there exists a constant $C=C(A,q)$ depending on $A$ and $q$ such that
    $$
\bigg(\dashint_A \Big|f-\dashint_A f\Big|^{q^*}\dd\LL^n\bigg)^{1/q^*}\le C \bigg(\dashint_A|\nabla  f|^q\dd\LL^n\bigg)^{1/q}\qquad\text{for every }f\in W^{1,q}(A),
    $$
    where $1/{q^*}=1/q-1/n$.
\end{defn}
We recall that  any ball in $\mathbb{R}^n$ supports Poincaré inequalities  \cite[Theorem 4.9]{EG15}.

\begin{rem}\label{rempoinc}
Let $A$ be a bounded open domain supporting Poincaré inequalities.
We recall the following fact: if $f\in W^{1,1}_\rmloc(A)$ is such that $\int_A|\nabla  f|^q\dd\LL^n<+\infty$, then $f\in L^{q^*}(A)$, where $1/q^*=1/q-1/n$. To show this, apply a Poincaré inequality to $f_m\defeq (f\wedge m)\vee -m\in W^{1,q}(A)$, with $\int_A|\nabla  f_m|^q\dd\LL^n\le \int_A|\nabla  f|^q\dd\LL^n$, and deduce that, for a constant $c_m\defeq\dashint_A f_m\dd\LL^n$, it holds
$$
\bigg(\dashint_A |f_m-c_m|^{q^*}\dd\LL^n\bigg)^{1/q^*}\le C \bigg(\dashint_A|\nabla  f|^q\dd\LL^n\bigg)^{1/q}.
$$
Now, if $B\subseteq A$ is a ball with $\bar B\subseteq A$, we have that 
$\Vert f_m\Vert_{L^1(B)}\le \Vert f\Vert_{L^1(B)}<+\infty$ and $\Vert f_m-c_m\Vert_{L^1(B)}$ is bounded in $m$, so that $\sup_m|c_m|<\infty$.
We also have that $\Vert f_m-c_m\Vert_{L^{q^*}(A)}$ is uniformly bounded.  Thus,  we infer that $\Vert f_m\Vert_{L^{q^*}(A)}$ is bounded in $m$, whence $f\in {L^{q^*}(A)}$.\fr
\end{rem}

\subsection{Distributions}

	We denote, as usual, $\DD(\Omega)=C^\infty_\rmc(\Omega)$ the space of test functions and  $\DD'(\Omega)$ its dual, i.e.\ the space of distributions \cite{schwartz1957theorie}. If $T\in \DD'(\Omega)$, we denote with $\nabla^2 T$ the distributional Hessian of $T$, i.e.\ the matrix of distributions $\{\partial^2_{i,j} T\}_{i,j\in 1,\dots,n}$ where $(\partial^2_{i,j} T )(f)\defeq T(\partial_i\partial_j f)$ for every $f\in\DD(\Omega)$. In a natural way, if $F\in \DD(\Omega)^{n\times n}$, we denote $$\nabla^2 T(F)\defeq\sum_{i,j=1\dots,n} \partial^2_{i,j}T(F_{i,j}).$$
\begin{rem}\label{rempartial}
Let $T$ be a distribution on $\Omega$ such that for every $i=1,\dots,n$, $\partial_i T$ is a Radon measure. Then $T$ is induced by a $\BV_\rmloc(\Omega)$ function.

The   proof of this fact is   classical.  Here,  we sketch it for the reader's convenience. 

We let $\{\rho_k\}_k$ be a sequence of Friedrich mollifiers. Let $B\subseteq\Omega$ be a ball such that $\bar B\subseteq\Omega$, so that, if $k$ is big enough (that we will implicitly assume in what follows), we have a well defined distribution $\rho_k\ast T$ on $B$, which is induced by a $C^\infty(\bar B)$ function, say $t_k$. It is immediate to show that for every $i=1,\dots,n$, $\int_B|\partial_i t_k|\dd\LL^n$ are uniformly bounded in $k$, as $T$ has derivatives that are Radon measures.
Therefore, using a Poincaré inequality on $B$, we have that for some $q^*>1$,
$\Vert t_k-c_k\Vert_{L^{q^*}(B)}$ is uniformly bounded in $k$, where $c_k\defeq\dashint_B t_k\dd\LL^n$. Hence, up to non-relabelled subsequences, $t_k-c_k$ converges to an $L^{q^*}(B)$ function $f$ in the   weak topology of $L^{q^*}(B)$ and then in the weak topology of  $\DD'(B)$. Also, $t_k$ converges in the topology of $\DD'(B)$ to $T$.
Hence $c_k=t_k-(t_k-c_k)$ converges in the weak topology of $\DD'(B)$ to $T-f\in\DD'(B)$. This forces $\{c_k\}_k\subseteq\RR$ to be bounded, so that also $t_k$ was bounded in $L^{q^*}(B)$ and hence  $T$  is induced by an $L^{q^*}(B)$ function on $B$. A partition of unity argument shows that $T$ is induced by an $L^1_\rmloc(\Omega)$ function, whence the conclusion.\fr
\end{rem}

\section{Hessian--Schatten Total Variation}\label{sec:HTV}
In this section, we fix $\Omega\subseteq\RR^n$ to be an open set and $p\in[1,+\infty]$. We let $p^*$ denote the conjugate exponent of $p$.
First,   we recall the definition of the HTV seminorm, presented in \cite{Aziznejad2021HSTV}, in the  spirit of the classical theory of functions of bounded variation. Next, we review some known results for the space of functions with bounded Hessian (see, \cite{Deme}), proposing at the same time a few refinements and/or extensions. 
\subsection{Definitions and Basic Properties}
\begin{defn}[Hessian--Schatten total variation]\label{defHSTV}
	Let $f\in L^1_\rmloc(\Omega)$. For every $A\subseteq\Omega$ open we define the Hessian-Schatten total variation of $f$ as 
	\begin{equation}\label{defd2}
		|\DIFF^2_p f|(A)\defeq \sup_F \int_A \sum_{i,j=1,\dots,n} f \partial_i\partial_j F_{i,j}\dd\LL^n,
	\end{equation}
	where the supremum runs among all $F\in C_\rmc^\infty(A)^{n\times n}$ with $\Vert F\Vert_{p^*,\infty}\le 1$.
	We say that $f$ has bounded $p$-Hessian--Schatten variation in $\Omega$ if $|\DIFF^2_p f|(\Omega)<\infty$. 
\end{defn}
\begin{rem}
	If $f$ has bounded $p$-Hessian--Schatten variation in  $\Omega$, then the set function defined in \eqref{defd2} is the restriction to open sets of a finite Borel measure, that we still call $	|\DIFF^2_p f|$. This can be proved with a classical argument, building upon \cite{DGL} (see also \cite[Theorem 1.53]{AFP00}).
	
	By its very definition, the $p$-Hessian--Schatten variation is lower semicontinuous with respect to $L^1_\rmloc$ convergence.\fr
\end{rem}

For any couple $p,q\in [1,+\infty]$, $f$ has bounded $p$-Hessian--Schatten variation if and only if $f$ has bounded $q$-Hessian--Schatten variation and moreover $$C^{-1}|\DIFF^2_p f|\le |\DIFF^2_q f|\le C|\DIFF^2_p f| $$
for some constant $C=C(p,q)$ depending only on $p$ and $q$.  
Hence, the induced topology is independent of the choice of $p$. 
For this reason, in what follows, we will often implicitly take $p=1$ (omitting thus to write $p$), and we will stress $p$ when this choice plays a role.

\bigskip

We prove now that having bounded Hessian--Schatten variation measure is equivalent to membership in $W^{1,1}_\rmloc$ with gradient with bounded total variation. Also, we compare the Hessian--Schatten variation measure to the total variation measure of the gradient. This will be a key observation, as it will allow us to use the classical theory of functions of bounded variation, see e.g.\ \cite{AFP00}.

\begin{prop}\label{hessiananddiffgrad}
Let $f\in L^1_{\rmloc}(\Omega)$. Then the following are equivalent:
\begin{enumerate}
    \item $f$ has bounded Hessian--Schatten variation in $\Omega$,
    \item $f\in W^{1,1}_\rmloc(\Omega)$ and $\nabla f\in \BV_\rmloc(\Omega)$ with $|\DIFF\nabla f|(\Omega)<\infty$.
\end{enumerate}  If this is the case, then, as measures,
\begin{equation}\label{comptv}
    |\DIFF^2_p f|=\bigg|\dv{\DIFF\nabla f}{|\DIFF\nabla f|}\bigg|_p |\DIFF\nabla f|.
\end{equation}
In particular, there exists a constant $C=C(n,p)$ depending only on $n$ and $p$ such that 
$$
C^{-1}|\DIFF \nabla f|\le |\DIFF^2_p f|\le C|\DIFF \nabla f|
$$
as measures.
\end{prop}
\begin{proof}
We divide the proof in two steps.\\
\medskip\\\textbf{Step 1}. We prove $1\Rightarrow 2$. Let $T\in\DD'(\Omega)$ denote the distribution induced by $f\in L^1_\rmloc(\Omega)$. For $i=1,\dots,n$, define $S_i\defeq \partial_i T\in\DD'(\Omega)$. By the fact that $f$ has bounded Hessian--Schatten variation in $\Omega$, we can apply Riesz Theorem and deduce that for every $j=1,\dots,n$, $\partial_j S_i$ is induced by a finite measure on $\Omega$. Indeed, if $\varphi\in C_c^\infty(\Omega)$, it holds
$$
\partial_j S_i(\varphi)=\int_\Omega f \partial_j\partial_i\varphi\le C\Vert\varphi\Vert_{\infty},
$$
where $C$ is independent of $\varphi$.
Then, by Remark \ref{rempartial}, $S_i$ is induced by an $L^1_\rmloc(\Omega)$ function,  which proves the claim.
\medskip\\\textbf{Step 2}. We prove $2\Rightarrow 1$ and \eqref{comptv}. 
First, we can write $\DIFF\nabla f=M \mu$, where $|M(x)|_p=1$ for $\mu$-a.e.\ $x\in\Omega$. Namely,  $$M=\dv{\DIFF\nabla f}{|\DIFF\nabla f|} \bigg|\dv{\DIFF\nabla f}{|\DIFF\nabla f|}\bigg|_p^{-1}\qquad\text{and} \qquad\mu= \bigg|\dv{\DIFF\nabla f}{|\DIFF\nabla f|}\bigg|_p|\DIFF\nabla f|.$$ This decomposition depends on $p$, but we will not make this dependence explicit.

Let $A\subseteq\Omega$ be open and let $F\in C_\rmc^\infty(A)^{n\times n}$ with $\Vert F\Vert_{p^*,\infty}\le 1$. Then
\begin{equation}\notag
 \int_A \sum_{i,j} f\partial_i\partial_j F_{i,j}=\int_A \sum_{i,j} M_{i,j} F_{i,j}\dd\mu\le\mu(A),
\end{equation}
so that $f$ has bounded $p$-Hessian-Schatten variation and $|\DIFF_p^2 f|\le \mu$ as measures on $\Omega$. 

We show now that $\mu(\Omega)\le |\DIFF_p^2 f|(\Omega)$. Fix now $\epsilon>0$. 
By Lusin's Theorem, we can find a compact set $K\subseteq\Omega$ such that $\mu(\Omega\setminus K)<\epsilon$ and
the restriction of $M$ to $K$ is continuous. Since 
$$
\sup_{|N|_{p^*}\leq 1}M(x)\,\cdot\,N=1\qquad \text{for every } x\in K,
$$
by the continuity of $M$ we can find a Borel function $N$ with finitely many values such that $|N(x)|_{p^*}\le 1$ for every $x\in\Omega$ 
and $M\,\cdot\, N\ge 1-\epsilon$ on $K$. Now we take $\psi\in C^\infty_\rmc(\Omega)$  with $\Vert\psi\Vert_{\infty}\le 1$ and we let $\{\rho_k\}_k$ be a sequence of Friedrich mollifiers. We consider (if $k$ is big enough) $\psi(\rho_k\ast N)\in C_\rmc^\infty(\Omega)$, which satisfies $\Vert \psi(\rho_k\ast N)\Vert_{p^*,\infty}\le 1$
on $\Omega$ (by convexity of the Schatten $p^*$-norm). Therefore,
$$
|\DIFF^2_p f|(\Omega)\ge \int_\Omega \sum_{i,j} M_{i,j} \psi(\rho_k\ast N_{i,j})\dd\mu\geq
\int_K \sum_{i,j} M_{i,j} \psi(\rho_k\ast N_{i,j})\dd\mu-\epsilon.
$$
We let $k\rightarrow \infty$, taking into account that $x\mapsto N(x)$ is continuous on $K$ and we recall that $\psi$ was arbitrary to infer that 
$$
|\DIFF^2_p f|(\Omega)\ge  \int_K \sum_{i,j} M_{i,j} N_{i,j}\dd\mu
-\epsilon\geq(1-\epsilon)\mu(K)-\epsilon \ge (1-\epsilon)\mu(\Omega)-2\epsilon.
$$
As $\epsilon>0$ was arbitrary, the proof is concluded as we have shown that $|\DIFF^2_p f|=\mu$.
\end{proof}
\begin{rem}

	One may wonder what happens if, instead of defining the Hessian--Schatten total variation only on $L^1_\rmloc$ functions, we define it on the bigger space of distributions, extending, in a natural way, \eqref{defd2} to distributions, i.e.\ interpreting the right hand side as $\sup_F \sum_{i,j=1}^n\partial_i\partial_j T(F_{i,j})=\sup_F \sum_{i,j=1}^n T(\partial_i\partial_jF_{i,j})$.
	
	It turns out that the difference is immaterial:  distributions with bounded Hessian--Schatten total variation are induced by $L^1_\rmloc$ functions, and, of course, the two definitions of $p$-Hessian--Schatten total variation coincide. This is proved exactly as in \textbf{Step 1} of the proof of Proposition \ref{hessiananddiffgrad}, using Remark \ref{rempartial} once more.\fr
\end{rem}

The following proposition is basically taken from \cite{Deme} and is a density (in energy) result akin to Meyers--Serrin Theorem.
\begin{prop}\label{myersserrin}
	Let $f\in L^1_\rmloc(\Omega)$. Then, for every $A\subseteq\Omega$ open, it holds
	$$
	|\DIFF^2_p f|(A)=\inf\left\{\liminf_k \int_A |\nabla^2 f_k|_p\dd\LL^n\right\}
	$$
	where the infimum is taken among all sequences $\{f_k\}_k\subseteq C^\infty(A)$ such that $f_k\rightarrow f$ in $L^1_\rmloc(A)$.
	If moreover $f\in L^1(A)$, the convergence in $L^1_\rmloc(A)$ above can be replaced by convergence in $L^1(A)$.
\end{prop}

\begin{proof}
	The $(\le)$ inequality is trivial by lower semicontinuity. The proof of the opposite inequality is due to a Meyers--Serrin argument, and can be obtained adapting \cite[Proposition 1.4]{Deme} (we know that $f\in W^{1,1}_\rmloc(\Omega)$ thanks to Proposition \ref{hessiananddiffgrad}). Notice that in the proof of \cite{Deme} Hilbert-Schmidt norms instead of Schatten norms are used. The proof can be adapted with no effort to any norm. Alternatively, one may notice that the result with Hilbert-Schmidt norms implies the result for any other matrix norm, thanks to the Reshetnyak continuity Theorem (see e.g.\ \cite[Theorem 2.39]{AFP00}), taking into account that $\DIFF\nabla f_k\rightarrow\DIFF \nabla f$ in the weak* topology and \eqref{comptv}.
\end{proof}

\bigskip

Now we show that Hessian--Schatten total variations decrease under the effect of convolutions, that is a a well-known property
in the $\BV$ context. 

\begin{lem}\label{mollif}
Let $f\in L^1_\rmloc(\Omega)$ with bounded Hessian--Schatten variation in $\Omega$. Let also $A\subseteq\RR^n$ open and $\epsilon>0$ with $B_\epsilon(A)\subseteq\Omega$. Then, if $\rho\in C_\rmc(\RR^n)$ is a convolution kernel with $\supp\rho\subseteq B_\epsilon(0)$, it holds
$$
|\DIFF_p^2 (\rho\ast f)|(A)\le|\DIFF_p^2 f|(B_\epsilon(A)).  
$$ 
\end{lem}
\begin{proof}
Let $F\in C_\rmc^\infty(A)^{n\times n}$ with $\Vert F\Vert_{p^*,\infty}\le 1$. We compute 
\begin{equation}\label{vsnojd}
	\int_A \sum_{i,j} (\rho\ast f)\partial_i\partial_j F_{i,j}=\int_A \sum_{i,j} f\partial_i\partial_j (\check{\rho}\ast F_{i,j}),
\end{equation}
where $\check{\rho}(x)\defeq\rho(-x)$. Notice that, defining the action of the mollification component-wise, $\check{\rho}\ast F\in C_\rmc^\infty(\Omega)$ (by the assumption on the support of $\rho$) with (by duality)
$$
|(\check{\rho}\ast F)(x)|_p=\sup_M M\,\cdot\,(\check{\rho}\ast F)(x)=\sup_M  (\check{\rho}\ast (M\,\cdot\, F))(x)\le (\check{\rho}\ast 1) (x)\le 1,
$$
where the supremum is taken among all $M\in\RR^{n\times n}$ with $|M|_{p^*}\le 1$. Here we used that $|F|_{p^*}(x)\le 1$ for every $x\in\Omega$.
Hence $\Vert (\check{\rho}\ast F)\Vert_{p^*,\infty}\le 1$. Also, $\check{\rho}\ast F$ is supported in $B_\epsilon(A)$, so that the right hand side of \eqref{vsnojd} is bounded by $|\DIFF^2_p f|(B_\epsilon(A))$ and the proof is concluded as $F$ was arbitrary.
\end{proof}

\bigskip

In the following proposition we obtain an analogue of the classical Sobolev embedding Theorems tailored for our situation.  Recall Definition \ref{defnPoin}.
\begin{prop}[Sobolev embedding]\label{sobo}
Let $f\in L^1_\rmloc(\Omega)$ with bounded Hessian--Schatten variation in $\Omega$. 
Then
\begin{alignat*}{3}
	&f\in L^{n/(n-2)}_{\rmloc}(\Omega)\cap W_\rmloc^{1,n/(n-1)}(\Omega)\qquad&&\text{if }n\ge 3,\\
	&f\in L^\infty_{\rmloc}(\Omega)\cap W^{1,2}_\rmloc(\Omega)\qquad&&\text{if }n= 2,\\
	&f\in L^\infty_{\rmloc}(\Omega)\cap W^{1,\infty}_\rmloc(\Omega)\qquad&&\text{if }n= 1\end{alignat*}
and, if $n=2$, $f$ has a continuous representative.

More explicitly, for every bounded domain $A\subseteq\Omega$ that supports Poincaré inequalities and $r\in[1,+\infty)$, there an affine map $g=g(A,f)$ such that, setting $\tilde f\defeq f-g$, it holds that
\begin{alignat}{3}
	\label{eq1}&\Vert \tilde f\Vert_{L^{n/(n-2)}(A)}+\Vert \nabla\tilde f\Vert_{L^{n/(n-1)}(A)}\le C(A)|\DIFF^2 f|(A)\qquad&&\text{if }n\ge 3,\\\label{eq2}
	&\Vert \tilde f\Vert_{L^r(A)}+\Vert \nabla \tilde f\Vert_{L^2(A)}\le C(A,r)|\DIFF^2 f|(A)\qquad&&\text{if }n= 2,\\\label{eq3}
	&\Vert \tilde f\Vert_{L^\infty(A)}+\Vert \nabla \tilde f\Vert_{L^\infty(A)}\le C(A)|\DIFF^2 f|(A)\qquad&&\text{if }n= 1.
\end{alignat}
\end{prop}
\begin{proof}

The case $n=1$ is readily proved by direct computation (as, if a domain of $\RR$ supports Poincaré inequality has to be an interval) so that in the following we assume $n\ge 2$. Also, recall that Proposition \ref{hessiananddiffgrad} states that $f\in W^{1,1}_\rmloc(\Omega)$ with $\nabla f\in\BV_\rmloc(\Omega)$. Therefore we can apply \cite[Proposition 3.1]{Deme} to have continuity of $f$ in the case $n=2$, which also implies $L^\infty_\rmloc(\Omega)$ membership.

As balls satisfy Poincaré inequalities, it is enough to establish the estimates of the second part of the claim to conclude. Fix then $A$ and $r$ as in the second part of the statement. 

Let now $\{f_k\}_k$ be given by Proposition \ref{myersserrin} for $f$ on $A$. Iterating  Poincaré inequalities, taking into account Remark \ref{rempoinc}, we obtain affine maps $g_k$ so that, setting $\tilde f_k\defeq f_k-g_k$, $\tilde{f_k}$ satisfies \eqref{eq1} or \eqref{eq2}, depending on $n$. Arguing as for Remark \ref{rempoinc}, we see that $g_k$ is bounded in $L^1(B)$ for any ball $B\subseteq A$. This implies that $g_k$ and $\nabla g_k$ are bounded in $L^\infty(A)$. Therefore, up to extracting a further non relabelled subsequence, $\tilde f_k$ converges in $L_\rmloc^1(A)$ to $f-g$, for an affine function $g$. Lower semicontinuity of the norms at the left hand sides of \eqref{eq1} or \eqref{eq2} allows us to conclude the proof.
\end{proof}
\begin{rem}[Linear extension domains]
Let $n=2$, we keep the same notation as for Proposition \ref{sobo}. Assume also that  $A$ has the following property: there exists an open set $V\subseteq\RR^2$ with $\bar A\subseteq V$ and a bounded linear map $E:W^{1,2}(A)\rightarrow W^{1,2}(V)$ satisfying, for every $u$ with bounded Hessian--Schatten variation (hence $u\in W^{1,2}(A)$ by Proposition \ref{sobo}):
\begin{enumerate}
    \item $E u=u$ a.e.\ on $A$,
    \item $E u$ is supported in $V$,
    \item $|\DIFF^2 E u|(V)\le C|\DIFF^2 u|(A)$ for some constant $C$.
\end{enumerate}
Then we show that \eqref{eq2} can be improved to
$$
\Vert \tilde f\Vert_{L^\infty(A)}+\Vert \nabla \tilde f\Vert_{L^2(A)}\le C|\DIFF_p^2 f|(A),
$$
where we possibly modified the constant $C$.

First, by \eqref{eq2} it holds that $\Vert \tilde f\Vert_{W^{1,2}(A)}\le C|\DIFF^2 f|(A)$. Now take $\psi\in C_c^\infty(\RR^2)$ with support contained in $V$ and such that $\psi=1$ on $A$.
Then we have 
$$|\DIFF^2 (\psi E\tilde f)|(V)\le C\big(|\DIFF^2 (E\tilde f)|(V)+\Vert E\tilde f\Vert_{W^{1,2}(V)} \big)\le C|\DIFF^2 \tilde f|(A).$$
Then we use the continuous representative of $\psi E\tilde f$  as in \cite[Proposition 3.1]{Deme} and, from its very definition, the claim follows.

It is easy to see that $(0,1)^2$ is suitable for the above argument, see Lemma \ref{extlemma} below and its proof.\fr
\end{rem}

\bigskip

The strict convexity of the Schatten $p$-norm, for $p\in(1,+\infty)$ has, as a consequence, the following rigidity result.
\begin{lem}[Rigidity]\label{rigidity} Let $f,g\in L^1_\rmloc(\Omega)$ with bounded Hessian--Schatten variation and assume that 
\begin{equation}\notag
	|\DIFF^2_p (f+g)|(\Omega)=	|\DIFF^2_p f|(\Omega)+	|\DIFF^2_p g|(\Omega).
\end{equation}
Then 
$$
|\DIFF^2_p (f+g)|=	|\DIFF^2_p f|+	|\DIFF^2_p g|
$$
as measures on $\Omega$. If moreover, $p\in (1,+\infty)$, then 
$$\DIFF\nabla f=\rho_f \DIFF\nabla (f+g)\qquad\text{and}\qquad\DIFF\nabla g=\rho_g \DIFF\nabla (f+g) $$ for a (unique) couple $\rho_f,\rho_g\in L^\infty(|\DIFF\nabla (f+g)|)$ such that $0\le \rho_f,\rho_g\le 1$ $|\DIFF\nabla (f+g)|$-a.e. and satisfying $\rho_f+\rho_g=1$ $|\DIFF\nabla (f+g)|$-a.e. In particular, for every $q\in [1,+\infty]$,
$$
|\DIFF^2_q (f+g)|=	|\DIFF^2_q f|+	|\DIFF^2_q g|
$$
as measures on $\Omega$.
\end{lem}

\begin{proof}
	The first claim follows from the triangle inequality and the equality in the assumption.
Now assume $p\in(1,+\infty)$. Take then $\rho_f$ and $\rho_g$, the Radon--Nikodym derivatives:
	$$
	|\DIFF^2_p f|=\rho_f	|\DIFF^2_p (f+g)|\qquad	\text{and}\qquad|\DIFF^2_p g|=\rho_g	|\DIFF^2_p (f+g)|
	$$
	as measures on $\Omega$, where $\rho_f+\rho_g=1 \ |\DIFF^2_p (f+g)|$-a.e.
	We can apply Proposition \ref{hessiananddiffgrad} and write the polar decompositions
	$\DIFF\nabla f=M_p|\DIFF_p^2 f|$, $\DIFF\nabla g=N_p|\DIFF_p^2 g|$ and  $\DIFF\nabla (f+g)=O_p|\DIFF\nabla (f+g)|$ where $|M_p|_p,|N_p|_p,|O_p|_p$ are  identically  equal to $1$. Therefore
	$\DIFF\nabla f=M_p\rho_f|\DIFF_p^2 (f+g)|$, $\DIFF\nabla g=N_p\rho_g|\DIFF_p^2 (f+g)|$ and  $\DIFF\nabla (f+g)=O_p|\DIFF_p^2 (f+g)|$ and by linearity we obtain that
	$$M_p\rho_f|\DIFF_p^2 (f+g)|+N_p\rho_g|\DIFF_p^2 (f+g)|=O_p|\DIFF_p^2 (f+g)|$$
which implies that $\rho_f M_p+\rho_g N_p=O_p\ |\DIFF_p^2 (f+g)|$-a.e. Taking $p$-Schatten norms, $$1=|O_p|_p=|\rho_f M_p+\rho_g N_p|_p\le \rho_f| M_p|_p+\rho_g| N_p|_p=1\qquad |\DIFF_p^2 (f+g)|\text{-a.e.}$$
which implies the claim by strict convexity. The last assertion is due to Proposition \ref{hessiananddiffgrad}.	
\end{proof}

\subsection{Boundary Extension}
\cite[Theorem 2.2]{Deme} provides us with an extension operator for bounded domains with $C^2$ boundary. However, we need the result for parallelepipeds. This can be obtained following  \cite[Remark 2.1]{Deme}. However, we sketch the argument as we are going also to need a slightly more refined result compared to the one stated in \cite{Deme}. This extension result (namely, its corollary Proposition \ref{bellapprox}) will play a key role in the proof of Theorem \ref{mainthm} below. 

\begin{lem}\label{parall}
	Let $\Omega=(a_0,a_1)\times\Omega'$ be a parallelepiped in $\RR^n$ and let $f\in L^1_\rmloc(\Omega)$ with bounded Hessian--Schatten variation in $\Omega$. Then, if we set $$\tilde{\Omega}\defeq(a_0-(a_1-a_0)/2,a_1)\times\Omega',$$ there exists $\tilde{f}\in L^1_\rmloc(\tilde{\Omega})$ with bounded Hessian--Schatten variation in $\tilde\Omega$ such that $\tilde{f}=f$ a.e.\ on $\Omega$,\begin{equation}\label{claim1}
		|\DIFF^2 \tilde{f}|(\{a_0\}\times \Omega')=0 
	\end{equation} and \begin{equation}\label{claim2}
		|\DIFF^2\tilde{f}|(\tilde{\Omega})\le C	|\DIFF^2{f}|(\Omega),
	\end{equation}
	where $C$ is a constant that does not depend on $f$.
\end{lem}
\begin{proof}
	Up to a linear change of coordinates, we can assume that $\Omega=(0,1)^n$. Set $\Omega_1=\Omega$ and $\Omega_2=(-1/2,0)\times(0,1)^{n-1} = M(\Omega)$, for $M(x,y)\defeq (-x/2,y)$, where we use coordinates $\RR\times \RR^{n-1}\ni(x,y)$ for $\RR^n$. Set also
	\begin{equation}\notag
		\tilde{f}(x,y)\defeq
		\begin{cases}
			f(x,y)\qquad&\text{if }(x,y)\in \Omega_1,\\
			3 f(-x,y)-2 f(-2x,y)\qquad&\text{if }(x,y)\in  \Omega_2.
		\end{cases}
	\end{equation}
	An application of the theory of traces (\cite[Theorem 3.87 and Corollary 3.89]{AFP00}) together with Proposition \ref{hessiananddiffgrad}  yields that $|\DIFF\nabla \tilde f|(\partial\Omega_1\cap\partial\Omega_2)=0$, hence \eqref{claim1}. Then, we compute
$$|\DIFF^2\tilde f|(\Omega_1\cup\Omega_2\cup (\partial\Omega_1\cap  \partial\Omega_2))= |\DIFF^2\tilde f|(\Omega_1)+|\DIFF^2\tilde f|(\Omega_2)\le C |\DIFF^2 f|(\Omega_1),$$
	where $C$ is a constant, so that \eqref{claim2} follows.
\end{proof}

\begin{lem}\label{extlemma}
	Let $\Omega=(0,1)^n$ and let $f\in L_\rmloc^1(\Omega)$ with bounded Hessian--Schatten variation in $\Omega$. Then there exist a neighbourhood 
	$\tilde{\Omega}$ of $\bar{\Omega}$ and $\tilde{f}\in L_\rmloc^1(\tilde{\Omega})$ with bounded Hessian--Schatten variation in $\tilde\Omega$ such that 
	\begin{equation}\label{claim3}
		|\DIFF^2 \tilde f|(\partial \Omega)=0
	\end{equation}
	and
	\begin{equation}\label{claim5}
	|\DIFF^2 \tilde f|(\tilde{\Omega})\le C	|\DIFF^2  f|(\Omega),
	\end{equation}
	where $C$ is a constant that does not depend on $f$.
\end{lem}
\begin{proof}
	Apply several times (a suitable variant) of Lemma \ref{parall}, extending $\Omega$ along each side. Notice that at each step, we are extending a parallelepiped which contains $\Omega$.
\end{proof}
\begin{prop}\label{bellapprox}
	Let $\Omega=(0,1)^n$ and let $f\in L_\rmloc^1(\Omega)$ with bounded Hessian--Schatten variation in $\Omega$. Then there exists a sequence $\{f_k\}_k\subseteq  C^\infty(\tilde{\Omega})$, where $\tilde{\Omega}$ is a neighbourhood of $\bar{\Omega}$ such that 
	\begin{equation}
		\begin{split}
			f_k&\rightarrow f\qquad\text{in }L^1(\Omega)\\
			|\DIFF^2_p f_k|(\Omega)&\rightarrow| \DIFF^2_p f|(\Omega)
		\end{split}
	\end{equation}
	for any $p\in [1,+\infty]$.
\end{prop}
\begin{proof}
	Take $\tilde{f}$ as in  Lemma \ref{extlemma} and, if $\{\rho_k\}_k$ is a sequence of Friedrich mollifiers, set $f_k\defeq\tilde{f}\ast\rho_k$. The claim follows from lower semicontinuity and Lemma \ref{mollif}.
\end{proof}

\section{A Density Result for {\rm CPWL} Functions}\label{sec:CPWL}
In this section, we study the density of  CPWL functions in the unit ball of the HTV functional.  As usual, we let $\Omega\subseteq\RR^n$ open and $p\in[1,+\infty]$.
\subsection{Definitions and The Main Result}
\begin{defn}
	We say that $f\in C(\Omega)$ belongs to $\rm CPWL(\Omega)$ if there exists a decomposition $\{P_k\}_k$ of $\RR^n$ in $n$-dimensional convex polytopes (with convex polytope we mean the closed convex hull of finitely many points), intersecting only at their boundaries such that for every $k$, $f_{|P_k\cap \Omega}$ is affine and such that for every ball $B$, only finitely many $P_k$ intersect $B$. 
\end{defn}
Notice that $\rm CPWL$ functions defined on bounded sets have automatically finite Hessian--Schatten variation, by Proposition \ref{hessiananddiffgrad}.

In the particular case $n=2$, we can and will assume that the convex polytopes $\{P_k\}_k$ as in the definition of ${\rm CPWL}$ function are triangles.

\begin{rem}\label{cpwlrem}
Let $f\in {\rm CPWL}(\Omega)$, where $\Omega\subseteq\RR^n$ is open.  Notice that $\nabla f$ is constant on each $P_k$, call this constant $a_k$.

 Thanks to Proposition~\ref{hessiananddiffgrad}, we can deal with $|\DIFF_p^2f|$ and $|\DIFF\nabla f|$ exploiting the theory of vector valued functions of bounded variation \cite{AFP00}. In particular, $|\DIFF\nabla f|$ will charge only $1$-codimensional faces of $P_k$. Then, take a non degenerate face $\sigma=P_k\cap P_{k'}$ for $k\ne k'$ (i.e.\ $\sigma$ is the common face of $P_k$ and $P_{k'}$). Then   the Gauss-Green Theorem gives $\DIFF\nabla f\mres \sigma=(a_{k'}-a_{k})\otimes \nu \HH^{n-1}\mres \sigma$, where $\nu$ is the unit normal to $\sigma$ going from $P_k$ to $P_{k'}$ (hence $(a_k-a_{k'})\perp\sigma$). Then,
\begin{equation}\label{eq:HTVcpwl}
|\DIFF^2 f|\mres \sigma =|(a_{k'}-a_{k})\otimes \nu|_1 \HH^{n-1}\mres \sigma=|a_{k'}-a_{k}|\HH^{n-1}\mres \sigma,
\end{equation}
where, as usual, $|a_{k'}-a_{k}|$ denotes the Euclidean norm.   Let us remark that \eqref{eq:HTVcpwl} has also been shown in \cite{Aziznejad2021HSTV},  directly relying on Definition~\ref{defHSTV}, which paved the way of developing numerical schemes for learning CPWL functions \cite{campos2021HTV,pourya2022delaunay}. 
Since $\dv{\DIFF\nabla f}{|\DIFF\nabla f|}$ has rank one $|\DIFF\nabla f|$-a.e.\ we obtain also
$$\bigg|\dv{\DIFF\nabla f}{|\DIFF\nabla f|}\bigg|_p=\bigg|\dv{\DIFF\nabla f}{|\DIFF\nabla f|}\bigg|=1\qquad|\DIFF\nabla f|\text{-a.e.}$$
\ (we recall that the matrix norm $|\,\cdot\,|$ without any subscript denotes the Hilbert-Schmidt norm).
It follows from \eqref{comptv} that $|\DIFF_p^2 f|=|\DIFF\nabla f|$ for every $p\in [1,+\infty]$, in particular, $|\DIFF_p^2 f|$ is independent of $p$.   

Notice also that the rank one structure of $\DIFF\nabla f$ is a particular case of the celebrated Alberti's theorem  \cite{alberti_1993}, for vector-valued $\rm BV$ functions. According to this theorem the rank one structure holds for the singular part of the distributional derivative.  \fr 
\end{rem}

\bigskip

The following theorem  on the density of CPWL functions is the main theoretical  result of this paper. Its proof is deferred to Section \ref{proofmain}.
In view of it, notice that by Lemma \ref{extlemma} together with Theorem \ref{sobo}, if $f\in L^1_\rmloc((0,1)^2)$ has bounded Hessian--Schatten variation in $(0,1)^2$, then $f\in L^\infty((0,1)^2)$. Also, notice that the statement of the theorem is for $p=1$ only. This will be discussed in the forthcoming Remark \ref{cdsmkocs}.

\begin{thm}\label{mainthm}
Let $n=2$, let $\Omega=(0,1)^2$ and let $p=1$. Then $\rm CPWL(\Omega)$ functions are dense in energy $|\DIFF_1^2\,\cdot\,|(\Omega)$ in $$\left\{f\in L^1_\rmloc(\Omega): f\text{ has bounded Hessian--Schatten variation}\right\}$$
with respect to the $L^\infty(\Omega)$ topology.   Namely, for any $f\in L^1_\rmloc(\Omega)$ with bounded Hessian-Schatten variation in $\Omega$
there exist $f_k\in \rm CPWL(\Omega)$ convergent in $L^\infty(\Omega)$ to $f$ with $|\DIFF^2_1 f_k|(\Omega)$ convergent to
$|\DIFF^2_1 f|(\Omega)$.
\end{thm}

\begin{rem}\label{cdsmkocs}
Theorem \ref{mainthm} shows in particular density in energy $|\DIFF^2_1\,\cdot\,|(\Omega)$ of ${\rm CPWL}(\Omega)$ functions with respect to the $L^1_\rmloc(\Omega)$ convergence. Notice that this conclusion is false if we take instead the $|\DIFF\nabla\,\cdot\,|(\Omega)$ seminorm,  and this provides one more theoretical justification of the relevance of the Schatten $1$-norm.

We now justify this claim. By Remark \ref{cpwlrem}, it is easy to realize that the two seminorms above coincide for ${\rm CPWL}(\Omega)$ functions, but are, in general, different for arbitrary functions. For example, take $f((x,y))\defeq \frac{x^2+y^2}{2}$. Then $|\DIFF^2_1 f|=2\mathcal{L}^2$,  whereas $|\DIFF\nabla f|=\sqrt{2}\mathcal{L}^2$. Now assume by contradiction that there exists a sequence $\{f_k\}_k\subseteq{\rm CPWL}(\Omega)$ such that $f_k\rightarrow f$ in $L^1_\rmloc(\Omega)$ and  $|\DIFF\nabla f_k|(\Omega)\rightarrow|\DIFF\nabla f|(\Omega)$. Then
\begin{align*}
    \liminf_k|\DIFF \nabla f_k|(\Omega)=\liminf_k|\DIFF^2_1 f_k|(\Omega)\ge |\DIFF^2_1 f|(\Omega)>|\DIFF\nabla f|(\Omega),
\end{align*}
which is absurd. This also gives the same conclusion for $|\DIFF_p^2\,\cdot\,|$, in the case $p\in (1,+\infty]$.\fr
\end{rem}

We conjecture that the result of Theorem \ref{mainthm} can be extended to arbitrary dimensions. 
\begin{cnj}\label{maincnj}
The density result of Theorem \ref{mainthm} remains valid when the input domain is chosen to be any $n$-dimensional hypercube, $\Omega=(0,1)^n$.\footnote{After that our result has been announced and the completion of this paper, Sergio Conti in Bonn announced the proof
of the conjecture}
\end{cnj}

\subsection{Proof of  Theorem \ref{mainthm}}\label{proofmain}
This whole section is devoted to the proof of Theorem \ref{mainthm}.  Remarkably,  our proof is constructive and provides an effective algorithm to build such approximating sequence. 

\bigskip

Take $f\in L^1_\rmloc(\Omega)$ with  finite Hessian--Schatten variation. 
We remark again  that indeed $f\in L^\infty(\Omega)$. We notice that we can assume with no loss of generality that $f$ is the restriction to $\Omega$ of a $C_\rmc^\infty (\RR^2)$ function. This is due to Proposition~\ref{bellapprox} (and its proof), a cut off argument and 
and a diagonal argument. Still, we only have to bound Hessian--Schatten variations only on $\Omega$. 

We want to find a sequence $\{f_k\}_k\subseteq{\rm CPWL}(\Omega)$ such that $f_k\rightarrow f$ in $L^\infty(\Omega)$ and $\limsup_{k} |\DIFF^2_1 f_k|(\Omega)\le |\DIFF^2_1 f|(\Omega)$. This will suffice, by lower semicontinuity.
\medskip
\\ \textbf{Step 1.} Fix now $\epsilon>0$ arbitrarily. The proof will be concluded if we find $g\in {\rm CPWL}(\Omega)$ with $\Vert f-g\Vert_{L^\infty(\Omega)}\le \epsilon$ and $|\DIFF^2_1 g|(\Omega)\le|\DIFF^2_1 f|(\Omega)+C_f\epsilon $, where $C_f$ is a constant that depends only on $f$ 
(via its derivatives, even of second and third order)  that still has to be determined. In what follows we will allow $C_f$ to vary from line to line.
\medskip
\\ \textbf{Step 2.} We add a bit of notation.
Let $v,w\in S^1$ with $v\perp w$, $s\in\RR^2$ and $h\in(0,\infty)$. We call $G(v,w,s,h)$ the grid of $\RR^2$
 $$G(v,w,s,h)\defeq \left\{s+t v+ z h w: t\in\RR,z\in\mathbb Z\right\}\cup \left\{s+z h v+ t w: t\in\RR,z\in\mathbb Z\right\}.$$
The grid consist in boundaries of squares (open or closed) that are called squares of the grid. Vertices of squares of the grid are called vertices of the grid and the same for edges. Notice that $G(v,w,s,h)$ contains a square with vertex $s$ and whose squares have sides of length $h$ and are parallel either to $v$ or to $w$.
\medskip
\\ \textbf{Step 3.} For $N\in\mathbb N$, we consider the grid $$G^N\defeq G((1,0),(0,1),0,2^{-N})$$ and we let $Q_k^N$ denote the closed squares of this grid that are contained in $\bar\Omega$. Here $k=1,\dots,2^{2 N}$.
\medskip
\\ \textbf{Step 4.} For every $N$ we find two collections of matrices $\{D_k^N\}_k$ and $\{U^N_k\}_k$  satisfying the following properties, for every $k$:
\begin{enumerate}
	\item \label{diag}$D_k^N$ is diagonal.
	\item\label{form} $U^N_k\in O(\mathbb Q^2)$ is a rotation matrix of angle $\theta_k\in (0,\pi/2)$, $\theta_k\ne \{\pi/4\}$.
	\item\label{apprx} It holds that
	$$
	\lim_{N\rightarrow \infty}\sup_k \sup_{x\in Q^N_k}| (U_k^N)^t\nabla^2 f(x) U_k^N -D_k^N|_1\rightarrow 0.
	$$
\end{enumerate}
To build such sequences, first build $\{ D_k^N\}_k$ and $\{\tilde U^N_k\}_k$ with $ D_k^N$ diagonal and $\tilde U^N_k\in O(\RR^2)$ such that 
\begin{equation}\label{nodscd}
	(\tilde U^N_k)^t\nabla^2 f(x_k^N)\tilde U^N_k= D_k^N,
\end{equation} where $x_k^N$ is the centre of the square $Q^N_k$. We can do this thanks to the symmetry of Hessians of smooth functions. 

We denote $R_\theta$ the rotation matrix of angle $\theta$.
We set $\hat U^N_k\defeq \tilde U_k^N A_k$, where  $A_k$  is a matrix of the type \begin{equation*}
\begin{pmatrix}
0&\pm 1\\
\pm 1&0
	\end{pmatrix}\qquad\text{or}\qquad
\begin{pmatrix}
	\pm 1&0\\
	0&\pm 1
\end{pmatrix}
\end{equation*}
defined in such a way that  $\hat U^N_k=R_{\hat\theta_k}$, for  some  $\hat\theta_k\in [0,\pi/2)$. Notice that \eqref{nodscd} still holds for $\hat U^N_k$ in place of $\tilde U^N_k$.

Now notice that points with rational coordinates are dense in $S^1\subseteq\RR^2$, as a consequence of the well known fact that the inverse of the stereographic projection maps $\mathbb Q$ into $\mathbb Q^2$. 
Therefore we can find $\theta_k\in (0,\pi/2)$, $\theta_k\ne\pi/4$ so close to  $\hat\theta_k$ so that $|R_{\theta_k}-R_{\hat \theta_k}|_1\le N^{-1}$ and such that $R_{\theta_k}\in\mathbb{ Q}^{2\times 2}$.
 Then, set $U_k^N\defeq R_{\theta_k}$. Items \eqref{diag} and \eqref{form} hold by the construction above, whereas item \eqref{apprx} can be proved taking into account also the smoothness of $f$.
 
 We write
 \begin{equation*}
 	U^N_k=
 	\begin{pmatrix}
 		|&|\\
 		v_k^N&w_k^N\\
 		|&|
 	\end{pmatrix}.
 \end{equation*}
Notice that  $v_k^N\perp w_k^N$ and $v_k^N,w_k^N\in S^1$. Also, $\theta_k$ is the angle formed by the $x$-axis with $v_k^N$ so that $\tan(\theta_k)=\sfrac{(v_k^N)_2}{(v_k^N)_1}\in\mathbb Q$ by \eqref{form}.
\medskip
\\ \textbf{Step 5.} 
By item \eqref{apprx} of \textbf{Step 4}, we take $N$ big enough so that \begin{equation}\label{csnjodsa}
	\sup_k \sup_{x\in Q^N_k}| (U_k^N)^t\nabla^2 f(x) U_k^N -D_k^N|\le \epsilon. 
\end{equation} We suppress the dependence on $N$ in what follows as from now $N$ will be fixed. Also, we can, and will, assume $2^{-N}\le\epsilon$.
\medskip
\\ \textbf{Step 6.} We consider grids on $Q_k$, for every $k$ and depending on $K\in\mathbb N$, free parameter. We recall that $Q_k$ has been defined in \textbf{Step 3}.
These grids will be called $$G^K_k\defeq G(v_k,w_k,s_k,h_k^K),$$ where $h_k^K$ will be determined in this step and $s_k$ is any of the vertices of $Q_k$ (the choice of the vertex will not affect the grid).

For every $k$, we write $$\mathbb Q\ni\tan(\theta_k)=\frac{q_k}{p_k},$$
where ${\rm MCD} (p_k,q_k)=1$. 
We define also $$h_k^K\defeq \frac{1}{2^N} \frac{1}{2^K}\frac{1}{ \sqrt{p_k^2+q_k^2}\prod_{h\ne k}q_h}.$$
Notice that
$$
U_k=\frac{1}{\sqrt{p_k^2+q_k^2}}
\begin{pmatrix}
	p_k&-q_k\\
	q_k&p_k
\end{pmatrix},
$$
and, as $U_k$ is an orthogonal matrix, we have that 
$$
{\sqrt{p_k^2+q_k^2}}U_k^{-1}=
\begin{pmatrix}
	p_k&q_k\\
	-q_k&p_k
\end{pmatrix}\in\mathbb Z^{2\times 2}.
$$
This ensures that the vertices of $Q_k$ are also vertices of $G_k^K$. Now notice that lines in $G_k^K$ parallel to $v_k$ intersect the horizontal edges of $Q_k$ in points spaced $\sfrac{h_k^K}{\sin(\theta_k)}$ and also lines in $G_k^K$ parallel to $w_k$ intersect the vertical edges of $Q_k$ in points spaced $\sfrac{h_k^K}{\sin(\theta_k)}$.
We now compute
$$
\frac{h_k^K}{\sin(\theta_k)}={h_k^K}{\sqrt{1+\cot^2(\theta_k)}}=\frac{1}{2^N 2^K \sqrt{p_k^2+q_k^2}\prod_{h\ne k}q_h}\sqrt{1+\frac{p_k^2}{q_k^2}}=\frac{1}{2^N}\frac{1}{ 2^K}\frac{1} {\prod_h q_h}
$$
and we notice that this quantity depends only on $K$ (and on $N$) but not on $k$.
\medskip
\\ \textbf{Step 7.} Now we want to build a triangulation for the square $Q_k$, such triangulation will depend on the free parameter $K$ and will be called $\Gamma_k^K$. We will then glue all the triangulations $\{\Gamma_k^K\}_k$ to obtain $\Gamma^K$, a triangulation for $\bar\Omega$. We call edges and vertices of triangulation the edges and vertices of its triangles.  We refer to Figure~\ref{fig} for an illustration of the proposed triangulation. 
\begin{figure}[t]
\begin{minipage}{1.0\linewidth}
  \centering
  \centerline{\includegraphics[width= \linewidth]{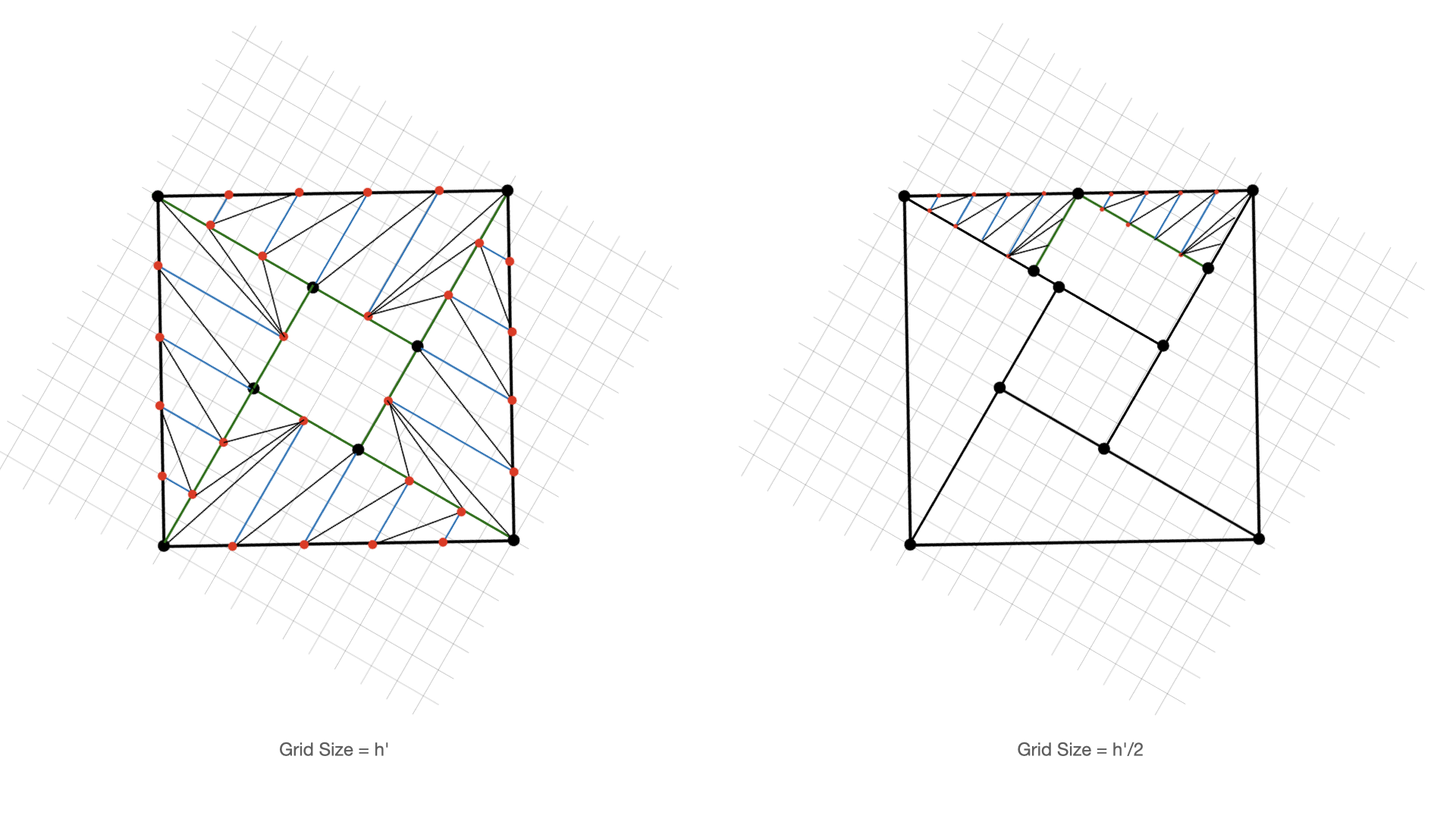}}
  \caption{  An illustration of the proposed triangulation in the square $Q_k$.}
  \label{fig} \medskip
\end{minipage}
\end{figure}

Fix for the moment $k$. We take the grid $G_k^0$. By symmetry, we can reduce ourselves to the case of  $\theta_k\in (\pi/4,\pi/2)$. Indeed, if $\theta_k\in (0,\pi/4)$, consider $\mathcal S$ to be the reflection against the axis passing through the top left  and bottom right vertex of $Q_k$, let $v_k'\defeq-\mathcal S v_k$ and $w_k'\defeq \mathcal S w_k$, build the triangulation $(\Gamma_k^K)'$ according to $v_k'$ and $w_k'$ and finally set $\Gamma_k^K\defeq\mathcal S(\Gamma_k^K)'$. 

Our building block for the triangulation is the triangle $\mathcal T^0_u$. We denote $A,B,C,D$ the vertices of the square $Q_k$, with $A$ corresponding to the top left vertex and the other named clockwise. Let $M,N,O,P$ denote the midpoints of $\overline{AB},\overline{BC},\overline{CD},\overline{DA}$ respectively. Then $\mathcal T^0_u=ABE$ is the right triangle with hypotenuse $\overline{AB}$ and such that its angle in $A$ is  $\pi/2-\theta_k$ and such that $E$ lies inside $Q_k$. We notice that $E$ is a vertex of $G_k^0$ by what proved in \textbf{Step 6}. Now we consider the intersections of lines of $G_k^0$ parallel to $v_k$ with the hypotenuse of $\mathcal T_u^0$ (these are not, in general, vertices of $G_k^0$) and the intersections of vertices of $G_k^0$ with the short sides of $\mathcal T_u^0$. 
Then we triangulate $\mathcal T^0_u$ in such a way that the vertices of the triangulation on the sides $\mathcal T^0_u$ are exactly at the points just considered.  Any finite triangulation is possible, but it has to be fixed. Now we rotate a copy of $\mathcal T^0_u$ (together with its triangulation) clockwise by $\pi/2$ and we translate it so that the point corresponding to $A$ moves to $B$. We thus obtain a triangulated triangle  $\mathcal T^0_r=BCF$. By construction, the triangulation on  $\mathcal T^0_r$ has the following property: its vertices on the hypotenuse of  $\mathcal T^0_u$ correspond to the intersection points of  lines of $G_k^0$ parallel to $w_k$ with the hypotenuse and its vertices on the short sides are exactly the vertices of $G_k^0$ on the short sides. Then we continue in this fashion to obtain four triangulated triangles, $\mathcal T^0_u,\mathcal T^0_r,\mathcal T^0_d,\mathcal T^0_l$,  as in the left side of Figure~\ref{fig}.  Notice that $\mathcal T^0_u\cup\mathcal T^0_r\cup\mathcal T^0_d\cup\mathcal T^0_l$, together with its triangulation is invariant by rotations of $\pi/2$ with centre the centre of $Q_k$.
Notice also that $Q_k\setminus \big(\mathcal T^0_u\cup\mathcal T^0_r\cup\mathcal T^0_d\cup\mathcal T^0_l\big)$ is formed by 
 a square which is itself a  union of squares, each with sides parallel to $v_k$ or $w_k$ and of length $h_k^0$.
 We triangulate $Q_k\setminus \big(\mathcal T^0_u\cup\mathcal T^0_r\cup\mathcal T^0_d\cup\mathcal T^0_l\big)$ in the standard way, where by standard way we mean the triangulation obtained considering the grid $G^K_k$ (now $K=0$) and, for every square of the grid, the diagonal with direction $(v_k-w_k)/\sqrt 2$. This is step $0$ and this triangulation will be called $\Gamma_k^0$. 

We show now how to build the triangulation at step $K+1$, $\Gamma_k^{K+1}$ starting from the one at step $K$, $\Gamma_k^K$, 
 see the right side of Figure~\ref{fig}.  At step $K$ we will have $\mathcal T^K_u,\mathcal T^K_r,\mathcal T^K_d,\mathcal T^K_l$. Now $\mathcal T^{K+1}_u$ will be union of two copies of 
$\mathcal T^{K}_u$ scaled by a factor $1/2$ but not rotated nor reflected, but translated so that the vertices corresponding to $A$ will correspond to $A$ and $M$ respectively. Also the triangulation of $\mathcal T^{K}_u$ is scaled and maintained. We do the same for $\mathcal T^{K}_r,\mathcal T^{K}_d,\mathcal T^{K}_l$, so that $\mathcal T^{K+1}_u\cup \mathcal T^{K+1}_r\cup \mathcal T^{K+1}_d\cup\mathcal  T^{K+1}_l$ together with its triangulation is invariant by rotations of $\pi/2$ with centre the centre of $Q_k$.\\
  We triangulate $Q_k\setminus \big(\mathcal T^{K+1}_u\cup\mathcal T^{K+1}_r\cup\mathcal T^{K+1}_d\cup\mathcal T^{K+1}_l\big)$ using the standard triangulation, with respect to $G_k^{K+1}$. We remark that $Q_k\setminus \big(\mathcal T^{K+1}_u\cup\mathcal T^{K+1}_r\cup\mathcal T^{K+1}_d\cup\mathcal T^{K+1}_l\big)$ is formed by union of squares, each with sides parallel to $v_k$ or $w_k$ and of length $h_k^{K+1}$.
  Notice that if $\sigma$ is a segment that is part of the boundary of one of $\mathcal T^{K+1}_u,\mathcal T^{K+1}_r,\mathcal T^{K+1}_d,\mathcal T^{K+1}_l$ and $\sigma$ is not contained in the boundary of $Q_k$, then the vertices of the triangulations on $\sigma$ coincide exactly with vertices of $G^{K+1}_k$ on $\sigma$, so that we have a well defined triangulation,  of $Q_k$ that we call $\Gamma_k^{K+1}$.
 
 Now we define $\Gamma^{K}$ as the triangulation of $\bar\Omega$ obtained by considering all the triangulations in $\{\Gamma_k^{K}\}_k$. Notice that, by \textbf{Step 6}, the triangulations in $\{\Gamma_k^{K}\}_k$ can be joined, as their vertices  on the boundaries of $\{Q_k\}_k$ match.
 Notice that for every $K$, $\mathcal T^{K}_u\cup\mathcal T^{K}_r\cup\mathcal T^{K}_d\cup\mathcal T^{K}_l$ is contained in a $2^{-N}2^{-K}$ neighbourhood of the lines of $G^N$, and this neighbourhood (in $\Omega$) has vanishing area as $K\rightarrow \infty$. Therefore, squares of the grid that are triangulated by $\Gamma^K$  in the standard way and such that also their eight neighbours are triangulated in the standard way by $\Gamma^K$ eventually cover monotonically $\Omega$, up to the axes of the grid $G^N$. Notice also that triangles in $\Gamma^K$ have edges of length smaller that $2^{-N} 2^{-K}$.
 
 We add here this crucial remark on which we will  heavily rely in the sequel and which well be the occasion to introduce the angle $\bar\theta$. There exists an angle, $\bar\theta>0$, such that every angle in the triangles of $\Gamma^K$ is bounded from below by $\bar\theta$, uniformly in $K$ ($\bar\theta$ depends on the choice of the various triangulations of $\mathcal T_u^K$, that, in turn, depend on $N$, so that $\bar\theta$ depends only on $N$ and $f$). This property is ensured by the self-similarity construction, that provides at each step $K$ two families of triangles, those arising from self-similarity and those arising from the bisection of
a (tilted) square with sides parallel to those of $Q_k$, as in Figure~\ref{fig}.
 \medskip
 \\ \textbf{Step 8.} For every $K$, we set $g^K$ as the $\rm CPWL$ interpolant  of $f$ according to $\Gamma^K$. Recall that $\rm CPWL$ functions on $\Omega$ have finite Hessian--Schatten total variation. We can compute $|\DIFF^2_1 g^K|=|\DIFF \nabla g^K|$ explicitly, that will be concentrated on jump points of the $\nabla g^K$, i.e.\ on the edges of the triangulation $\Gamma^K$ (Remark \ref{cpwlrem}).
 
 The computations of \textbf{Step 9} below ensure that $\{g^K\}_K$ are equi-Lipschitz functions, so that it is clear that as $K\rightarrow\infty$ it holds that $\Vert f-g^K\Vert_{L^\infty(\Omega)}\rightarrow 0$. We claim that $\limsup_{K\to\infty}|\DIFF^2_1 g^K|(\Omega)\le |\DIFF ^2 f|(\Omega)+C_f\epsilon$. Let $U^\delta$ denote the open $\delta$ neighbourhood of $G^N$, intersected with $\Omega$. 
 
 Recall the definition of $\bar\theta$ given at the end of \textbf{Step 7}. Some of our estimates depend
 on $\bar\theta$ (see,  in particular,  the first item below and \textbf{Step 9}) whose value essentially depends on $N$. Since $N$ has been fixed, depending on $\epsilon$
 and the modulus of continuity of $\nabla^2 f$, we may absorb the $\bar\theta$ dependence into the $f$ dependence.
 
 The claim, hence the conclusion, will be a consequence of these two following facts, stated for $T$ closed triangle in $\Gamma^K$, say $T\in Q_{\bar k}$:
 \begin{enumerate}
 	\item\label{it1}it holds  $$|\DIFF^2_1 g^K|(T\cap\Omega)\le C_f\LL^2(T);$$
 	\item\label{it2} whenever $T$ does \emph{not} intersect $U^{2\cdot2^{-N} 2^{-K}}$,
 	then $$\frac 12 |\DIFF^2_1 g^K|(T)\le  (|D_{\bar k}|_1+C_f\epsilon)\LL^2(T).$$
 \end{enumerate}
Notice that in the first item we have a constant $C_f$ which depends on $f$, and hence we have to take $K$ big enough so that the contributions of these terms are small enough. 

Recall that in our estimates we allow $C_f$ to vary line to line.

We defer the proof of items \ref{it1} and \ref{it2} to \textbf{Step 9} and \textbf{Step 10} respectively, now let us show how to conclude the proof using these facts. 
Fix for the moment $K$ and $k$.
Now consider $\{T_i\}_i$, the (finite) collection (depending on $K$ and $k$, but we will not make such dependence explicit) of all the  closed triangles in the triangulation $\Gamma^K$ that are contained in $\bar{Q}_k$. Notice that
\begin{enumerate}[label=\roman*)]
	\item The interiors of $\{T_i\}_i$ are pairwise disjoint.
	\item If $\sigma$ is an edge of $\Gamma^K$ that  lies on the boundary of $Q_k$, then there exist exactly one element of $\{T_i\}_i$ having $\sigma$ as edge. This is due to the fact that we are taking triangles contained in $\bar Q_k$
	\item If $\sigma$ is an edge of $\Gamma^K$ that does \emph{not} lie on the boundary of $Q_k$, then there exist exactly two elements of $\{T_i\}_i$ having $\sigma$ as edge. 
\end{enumerate}

We order the collection  $\{T_i\}_i$  in such a way that 
$T_1,\dots, T_{\bar{i}}$ are contained in $U^{4\cdot 2^{-N}2^{-K}}$ and $T_{\bar{i}+1},\dots, $ do not intersect  $U^{2\cdot 2^{-N}2^{-K}}$. We compute, using items \ref{it1} and \ref{it2}, recalling iii) above for what concerns the factor ${1}/{2}$ in the first line, 
\begin{align*}
|\DIFF^2_1 g^K|(\Omega\cap \bar{Q}_k)&\le  \sum_{i\le \bar i} |\DIFF^2_1 g^K|(T_i)+\frac{1}{2}\sum_{i> \bar i} |\DIFF^2_1 g^K|(T_i)
\\&\le \sum_{i\le \bar i} C_{f}\mathcal{L}^2(T_i)+\sum_{i> \bar i}(|D_{ k}|_1+C_f\epsilon) \mathcal{L}^2(T_i)
\\&\le C_{f}\mathcal{L}^2(U^{4\cdot 2^{-N}2^{-K}}\cap Q_{k})+(|D_{ k}|_1+C_f\epsilon) \mathcal{L}^2(Q_k).
\end{align*}
Therefore, repeating the procedure for every $k$,
\begin{align*}
|\DIFF^2_1 g^K|(\Omega)&\le\sum_k |\DIFF^2_1 g^K|(\Omega\cap \bar{Q}_k)\\&\le 
\sum_kC_{f}\mathcal{L}^2(U^{4\cdot 2^{-N}2^{-K}}\cap Q_k)+C_f\epsilon \sum_k  \mathcal{L}^2(Q_k)+ \sum_k |D_k|_1 \mathcal{L}^2(Q_k)\\&\le
C_{f}\mathcal{L}^2(U^{4\cdot 2^{-N}2^{-K}})+ C_f\epsilon \mathcal{L}^2(\Omega)+\sum_k  |D_k|_1 \mathcal{L}^2(Q_k).
\end{align*}
Fix now $K$ big enough so that $C_{f}\mathcal{L}^2(U^{4\cdot 2^{-N}2^{-K}})\le \epsilon$, we have
$$|\DIFF^2_1 g^K|(\Omega)\le C_f\epsilon+\sum_k  |D_k|_1 \mathcal{L}^2(Q_k). $$
Now we compute, for every $k$, taking into account $\eqref{csnjodsa}$,
\begin{align*}
	|D_k|_1\mathcal{L}^2(Q_k)&=
	\int_{Q_k} |D_k|_1\le \int_{Q_k} \big(| (U_k^N)^t\nabla^2 f(x) U_k^N|_1+\epsilon\big)= \int_{Q_k} \big(| \nabla^2 f(x) |_1+\epsilon\big)\\&=|\DIFF^2_1 f|(Q_k)+\epsilon\mathcal{L}^2(Q_k)
\end{align*}
so that we can continue our previous computation to see that 
\begin{align*}
	|\DIFF^2_1 g^K|(\Omega)&\le C_f\epsilon+\sum_k  |D_k|_1 \mathcal{L}^2(Q_k)\le C_f\epsilon +\sum_k |\DIFF ^2 f|(Q_k)+\sum_k\epsilon \mathcal{L}^2(Q_k) \\&=C_f\epsilon +|\DIFF^2_1 f|(\Omega)+\epsilon\mathcal{L}^2(\Omega)\le C_f\epsilon +|\DIFF^2_1 f|(\Omega)
\end{align*}
thus concluding the proof.
 \medskip
\\ \textbf{Step 9.} We prove item \ref{it1} of \textbf{Step 8}. For definiteness, assume that $K$ is fixed. We will heavily use Remark \ref{cpwlrem} with no reference.

Say $T=ABC\subseteq Q_k$. It is enough to show that 
$|\DIFF^2_1 g^K|(\overline{AB})\le C_f\mathcal{L}^2(T)$, under the assumption that $\overline{AB}$ does not lie in the boundary of $\Omega$, so that
there exists another triangle $T'=ABC'$ of $\Gamma^K$ (possibly inside
an adjacent cube to $Q_k$, recall also that the mesh size parameter $K$ is independent of $k$), so that $T$ and $T'$ have disjoint interiors. 

Call $a=\nabla g^K$ on $T$ and $a'=\nabla g^K$ on $T'$.
Then,
\begin{alignat*}{2}
        &\begin{cases}
        a\,\cdot\,(B-C)=f(B)-f(C)\\
        a\,\cdot\,(A-C)=f(A)-f(C)
    \end{cases}
&&\qquad\text{and}\qquad
        \begin{cases}
        a'\,\cdot\,(B-C')=f(B)-f(C')\\
        a'\,\cdot\,(A-C')=f(A)-f(C').
    \end{cases}
\end{alignat*}
The mean value theorem gives 
\begin{alignat}{2}\label{asdcs}
    &\begin{pmatrix}
	(B-C)^t\\
	(A-C)^t
    \end{pmatrix}
    a&&=
    \begin{pmatrix}
	\nabla f(C)(B-C)+\frac{1}{2}(B-C)^t\nabla^2 f(\xi_1)(B-C)\\
	\nabla f(C)(A-C)+\frac{1}{2}(A-C)^t\nabla^2 f(\xi_2)(A-C)
    \end{pmatrix},
\end{alignat}
where $\xi_1,\xi_2\in T$.
Now notice that as the angles of $ABC$ are bounded below by $\bar\theta$, the matrix $$ \begin{pmatrix}
	(B-C)^t\\
	(A-C)^t
    \end{pmatrix}$$ is invertible and its inverse has norm bounded above by 
    $\frac{c_{\bar \theta}}{|\overline{AB}|}$, for a suitable constant $c_{\bar\theta}$. Also, possibly choosing a larger constant $c_{\bar\theta}$, the bound from below of the angles yields that $|\overline{BC}|\le c_{\bar\theta}|\overline{AB}|$ and $|\overline{AC}|\le c_{\bar\theta}|\overline{AB}|$. Similar considerations hold for the triangle $T'$. As $c_{\bar\theta}$ depends only on $\bar\theta$, we will absorb this dependence into the $f$ dependence, as announced above.
    
    We rewrite then \eqref{asdcs} as
 \begin{alignat*}{2}
 a=\nabla f(C)+\frac{1}{2}
 \begin{pmatrix}
	(B-C)^t\\
	(A-C)^t
    \end{pmatrix}^{-1}
    \begin{pmatrix}
	(B-C)^t\nabla^2 f(\xi_1)(B-C)\\
	(A-C)^t\nabla^2 f(\xi_2)(A-C)
    \end{pmatrix}.
\end{alignat*}   
Similarly,
 \begin{alignat*}{2}
 a'=\nabla f(C')+\frac{1}{2}
 \begin{pmatrix}
	(B-C')^t\\
	(A-C')^t
    \end{pmatrix}^{-1}
    \begin{pmatrix}
	(B-C')^t\nabla^2 f(\eta_1)(B-C')\\
	(A-C')^t\nabla^2 f(\eta_2)(A-C')
    \end{pmatrix},
\end{alignat*}   
for $\eta_1,\eta_2\in T'$.
Hence
\begin{align*}
    |\DIFF^2_1 g^K|(\overline{AB})=|a-a'||\overline{ AB}|\le\bigg( |\nabla f(C)-\nabla f(C')|+\frac{C_{f}}{|\overline{AB}|}|\overline{AB}|^2\bigg)|\overline{AB}|.
\end{align*}
Now, $|\nabla f(C)-\nabla f(C')|\le\max|\nabla f| (|\overline{AC}|+|\overline{AC'}|)$ so that
$$
|\DIFF^2_1 g^K|(\overline{AB})\le C_{f}|\overline{AB}|^2
$$
and the right hand side is bounded above by $C_{f}\mathcal{L}^2(T)$ as the angles of $T$ are bounded below by $\bar\theta$.
\medskip
\\ \textbf{Step 10.} We prove item \ref{it2} of \textbf{Step 8}.  For definiteness, assume that $K$ and $k$ are fixed, for $T\subseteq Q_k$. We will heavily use Remark \ref{cpwlrem} with no reference again. Notice that $T$ lies in a closed square of $G_k^K$ and this square, together with the other squares of $G_k^K$ intersecting it (at the boundary), is triangulated in the standard way, by the assumption that $T$ does not intersect $U^{2\cdot 2^{-N}2^{-K}}$. Notice that the square mentioned before is divided by $\Gamma^K$ into two triangles. For definiteness, assume that $T$ is the one whose barycentre  has smaller $y$ coordinate, the other case being similar.
Also, for definiteness, assume that $\theta_k \in(\pi/4,\pi/2) $, the case $\theta_k\in(0,\pi/4) $ being similar. 

Call $T=ACD$, such that the angles are named clockwise and the angle at $D$ is of $\pi/2$. Call $B$ the other vertex of the square of the grid in which $T$ lies. Call $E$ the vertex of $\Gamma^K$ such that $C=(B+E)/{2}$. Call $a=\nabla g^K$ on $T$, $a'=\nabla g^K$ on $ACB$ and $a''=\nabla g^K$ on $CDE$. Finally, call $F\defeq (B+D)/{2}$ and $\ell=|\overline{AD}|$.   We refer to Figure \ref{Fig:Illust} for an illustration on the introduced notations. 
\begin{figure}[t]
\begin{minipage}{1.0\linewidth}
  \centering
  \centerline{\includegraphics[width= 0.7\linewidth]{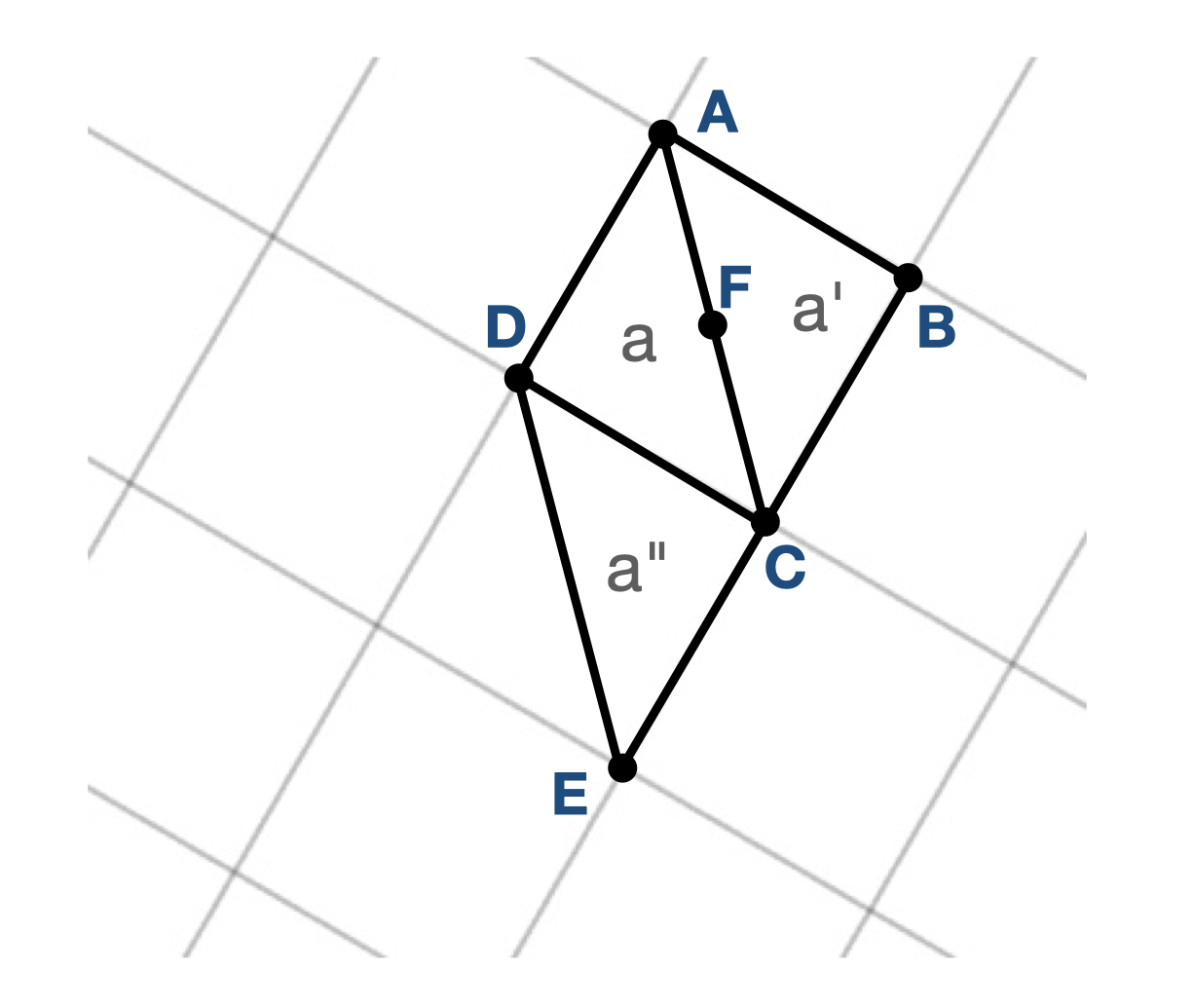}}
  \caption{  An illustration of the notations introduced in the Step 10 of the proof.}
  \label{Fig:Illust} \medskip
\end{minipage}
\end{figure}

We first estimate $|\DIFF^2_1 g^K|(\overline{AC})$: 
\begin{align*}
|\DIFF^2_1 g^K|(\overline{AC})= |a-a'|\HH^1(\overline{AC})=\sqrt{2}\ell  |a-a'|.
\end{align*}
Now we compute
\begin{align*}
	(g^K(D)-g^K(F))-(g^K(F)-g^K(B))&=f(D)+f(B)-f(A)-f(C)\\&=(f(D)-f(A))-(f(C)-f(B)).
\end{align*}
Now, 
\begin{align*}
\big|f(D)-f(A)-\big(-\ell\partial_{w_k} f(A)+\frac{\ell^2}{2}\partial^2_{w_k,w_k}f(A)\big)\big|&\le\frac{\ell^3}{6} \Vert \partial_{w_k,w_k,w_k}^3 f\Vert_\infty,\\
\big|f(C)-f(B)-\big(-\ell\partial_{w_k} f(B)+\frac{\ell^2}{2}\partial^2_{w_k,w_k}f(B)\big)\big|&\le\frac{\ell^3}{6} \Vert \partial_{w_k,w_k,w_k}^3 f\Vert_\infty.
\end{align*}
and 
\begin{align*}
	|\partial^2_{w_k,w_k}f(B)-\partial^2_{w_k,w_k}f(A)|&\le \ell \Vert \partial_{v_k,w_k,w_k}^3 f\Vert_\infty,\\
|\partial_{w_k}f(B)-\partial_{w_k}f(A)-\ell\partial^2_{v_k,w_k}f(A)|&\le \frac{\ell^2}{2} \Vert\partial_{v_k,v_k,w_k}^3 f\Vert_\infty.
\end{align*}

Then we can compute
\begin{align*}
|a-a'|=	\frac{\sqrt{2}}{\ell}\big|(g^K(D)-g^K(F))-(g^K(F)-g^K(B))\big|\le \frac{\sqrt{2}}{\ell}\big(\ell^2|\partial^2_{v_k,w_k}f(A)|+C_f\ell^3\big).
\end{align*}
All in all, recalling $2^{-N}\le\epsilon$, $$|\DIFF^2_1 g^K|(\overline{AC})\le 2 \ell^2\big(|\partial^2_{v_k,w_k}f(A)|+C_f \epsilon\big).$$
Now
$$ \partial^2_{v_k,w_k}f(A)=w_k^t\nabla^2 f(A)v_k=(0,1)^t U_k^t\nabla^2 f(A) U_k(1,0)$$
so that, by \eqref{csnjodsa}, 
$$
|\partial^2_{v_k,w_k}f(A)|\le |(0,1)^tD_k (1,0)|+ |(0,1)^t (U_k^t\nabla^2 f(A) U_k-D_k)(1,0)|\le C_f\epsilon
$$
and this gives
$$|\DIFF^2_1 g^K|(\overline{AC})\le\ell^2 C_f\epsilon.$$

We turn to $|\DIFF^2_1 g^k|(\overline{CD})$:
$$
|\DIFF^2_1 g^K|(\overline{CD})=|a-a''|\HH^1(\overline{CD})=\ell |a-a''|.
$$
Now we compute
\begin{align*}
(g^K(E)-g^K(C))-(g^K(D)-g^K(A))=(f(E)-f(C))+(f(A)-f(D)).
\end{align*}
Now
\begin{align*}
	\big|f(E)-f(C)-\big(-\ell\partial_{w_k} f(C)+\frac{\ell^2}{2}\partial^2_{w_k,w_k}f(C)\big)\big|&\le\frac{\ell^3}{6} \Vert \partial_{w_k,w_k,w_k}^3 f\Vert_\infty,\\
	\big|f(A)-f(D)-\big(\ell\partial_{w_k} f(D)+\frac{\ell^2}{2}\partial^2_{w_k,w_k}f(D)\big)\big|&\le\frac{\ell^3}{6} \Vert \partial_{w_k,w_k,w_k}^3 f\Vert_\infty.
\end{align*}
and 
\begin{align*}
	|\partial_{w_k} f(C)-\partial_{w_k} f(D)-\ell\partial^2_{v_k,w_k}f(D)|\le \frac{\ell^2}{2}\Vert \partial^3_{v_k,v_k,w_k}f\Vert_\infty.
\end{align*}

Then we can compute
\begin{align*}
	|a-a''|&=	\frac{1}{\ell}\big|(g^K(E)-g^K(C))-(g^K(D)-g^K(A))\big|\\&\le\frac{1}{\ell} \big( \frac{\ell^2}{2}|\partial^2_{w_k,w_k}f(C)|+\frac{\ell^2}{2}|\partial^2_{w_k,w_k}f(D)|+\ell^2|\partial^2_{v_k,w_k}f(D)|+C_f\ell^3\big).
\end{align*}
All in all, recalling again $2^{-N}\le \epsilon$,
$$|\DIFF^2_1 g^K|(\overline{CD})\le\ell^2(\frac{1}{2}|\partial^2_{w_k,w_k}f(C)|+\frac{1}{2}|\partial^2_{w_k,w_k}f(D)|+|\partial^2_{v_k,w_k}f(D)|+C_f\epsilon).$$
As before, $|\partial^2_{v_k,w_k}f(D)|\le C_f\epsilon$. Also, with similar computations as before,
$$
|\partial^2_{w_k,w_k}f(C)|\le |(0,1)^tD_k (0,1)|+ |(0,1)^t (U_k^t\nabla^2 f(D) U_k-D_k)(0,1)|\le |(0,1)^tD_k (0,1)|+ C_f\epsilon,
$$
and similarly 
$$
|\partial^2_{w_k,w_k}f(D)|\le |(0,1)^tD_k (0,1)|+ |(0,1)^t (U_k^t\nabla^2 f(D) U_k-D_k)(0,1)|\le |(0,1)^tD_k (0,1)|+ C_f\epsilon.
$$
Therefore,
$$|\DIFF^2_1 g^K|(\overline{CD})\le\ell^2\big( |(0,1)^tD_k (0,1)|+C_f\epsilon\big).$$
With similar computations we arrive at 
$$|\DIFF^2_1 g^K|(\overline{AD})\le\ell^2\big( |(1,0)^tD_k (1,0)|+C_f\epsilon\big).$$

Summing all the three contributions,
\begin{align*}
	|\DIFF^2_1 g^K|(T)&=|\DIFF^2_1 g^K|(\overline{AC})+|\DIFF^2_1 g^K|(\overline{CD})+|\DIFF^2_1 g^K|(\overline{AD})\\&\le\ell^2 C_f\epsilon
	+\ell^2\big( |(0,1)^tD_k (0,1)|+C_f\epsilon\big)+\ell^2\big( |(1,0)^tD_k (1,0)|+C_f\epsilon\big)
	\\&\le\ell^2 (C_f\epsilon +  |(0,1)^tD_k (0,1)|+ |(1,0)^tD_k (1,0)|)
	\\&=2\mathcal{L}^2(T) (C_f\epsilon +|D_k|_1)
\end{align*}
which concludes the proof.\null\nobreak\hfill\ensuremath{\square}

\section{Extremal Points of The Unit Ball}\label{sec:extreme}

Let $\Omega\defeq (0,1)^n\subseteq\RR^n$. 
In this section, we investigate the extremal points of the set
$$
\{f\in L^1_\rmloc(\Omega):|\DIFF^2 f|(\Omega)\le 1\}.
$$
Notice that elements of the set above are indeed in $L^1(\Omega)$, by Proposition \ref{sobo}, as cubes support Poincaré inequalities. In order to carry out our investigation, we will consider a suitable quotient space. We describe now our working setting.

We consider the Banach space $L^1(\Omega)$, endowed with the standard $L^1$ norm. We let $\mathcal A\subseteq L^1(\Omega)$ denote the (closed) subspace of affine functions. Therefore, ${L^1(\Omega)}/{\mathcal A}$,  endowed with the quotient norm, is still a Banach space. We call $\pi:L^1(\Omega)\rightarrow {L^1(\Omega)}/{\mathcal A}$ the canonical projection.
We define 
$$
\BB\defeq\left\{g\in \displaystyle {L^1(\Omega)}/{\mathcal A}: |\DIFF^2 g|(\Omega)\le 1\right\},
$$
where we notice that the $|\DIFF^2 \,\cdot\,|(\Omega)$ seminorm factorizes to the quotient, so that $\BB=\pi(\{f\in L^1(\Omega):|\DIFF^2 f |(\Omega)\le 1\})$. Also, by Proposition \ref{sobo} and standard functional analytic arguments (in particular, the Rellich--Kondrachov Theorem), we can prove that the convex set $\BB$ is compact. We will then be able to apply the Krein--Milman Theorem, for $\MM\subseteq\BB$: 
\begin{equation}\tag{KM}\label{KM}
   \BB={\rm \overline{co}(\MM)}\text{ if and only if }{\rm ext}(\BB)\subseteq\overline\MM. 
\end{equation}
We set 
$$\EE\defeq \pi({\rm CPWL}(\Omega))\cap {\rm ext}(\BB)\subseteq\SSS,$$
where
$$
\SSS\defeq\left\{g\in \displaystyle{L^1(\Omega)}/{\mathcal A}: |\DIFF^2 g|(\Omega)= 1\right\}.
$$
Thus, $\BB$ corresponds to the unit ball with respect to the $|\DIFF^2 \,\cdot\, |(\Omega)$ norm whereas $\mathcal{S}$ to the unit spere with respect to the same norm.

\bigskip

Even though we do not have an explicit characterization of extremal points of $\BB$, it is easy to establish whether a
function $g\in\pi(\rm CPWL(\Omega))$ is extremal or not.
\begin{prop}[CPWL Extreme Points]\label{fivedotfive}  A function
$g\in \pi(\rm{CPWL}(\Omega))\cap\SSS$ belongs to $\EE$ if and only if $h\in {\rm span}(g)$ for all 
$h\in\BB$ with $\supp(|\DIFF^2 h|)\subseteq\supp(|\DIFF^2 g|)$.
\end{prop}
\begin{proof}
The  ``only if'' implication follows easily from Proposition \ref{rigidity}. 

We prove now the converse implication via a perturbation argument, recall Remark \ref{cpwlrem}: we will use the same notation.

Let $g\in\EE$ and let $h\in\BB$ with $\supp(|\DIFF^2 h|)\subseteq\supp(|\DIFF^2 g|)$. We have to prove that $h\in{\rm span}(g)$. Assume by contradiction that $h\notin {\rm span}(g)$. We call now $\{P^g_k\}_k$ (resp.\ $\{P^h_k\}_k$) the triangles associated to $g$ (resp.\ $h$) and  $\{a_k^g\}_k$ (resp.\ $\{a_k^h\}_k$) the values associated to $\nabla g$ (resp.\ $\nabla h$). As we are assuming $\supp(|\DIFF^2 h|)\subseteq\supp(|\DIFF^2 g|)$, we can and will assume that $\{P^g_k\}_k$ and $\{P^h_k\}_k$ have the same cardinality and  $P^g_k=P_k^h$ for every $k$, so that we will drop the superscripts $g$ and $h$ on these triangles. Also, we assume that for every $k$, $P_k\subseteq\bar\Omega$.
Call 
$$
\delta\defeq\min\left\{| a^g_k-a^g_\ell|: \HH^1(\partial P_k\cap\partial P_\ell)>0,\ a^g_k\ne a^g_\ell\right\}
$$
and 
$$
\Delta\defeq\max\left\{| a^h_k-a^h_\ell|: \HH^1(\partial P_k\cap\partial P_\ell)>0\right\}
$$
and set finally $\epsilon \defeq {\delta}/{\Delta}$ (if $\Delta=0$, then $h=0$ and hence there is nothing to prove).
Now we write 
$$
g_1\defeq \frac{g+\epsilon h}{|\DIFF^2 (g+\epsilon h)|(\Omega)}\qquad\text{and}\qquad g_2\defeq \frac{g-\epsilon h}{|\DIFF^2 (g-\epsilon h)|(\Omega)},
$$
notice that $g_1,g_2\in\SSS$ are well defined as we are assuming $h\notin {\rm span}(g)$.
Clearly $g=c_1 g_1+c_2 g_2$, where $$c_1\defeq \frac{|\DIFF^2 (g+\epsilon h)|(\Omega)}{2} \qquad\text{and}\qquad 
c_2\defeq \frac{|\DIFF^2 (g-\epsilon h)|(\Omega)}{2}.$$
If we show that $c_1+c_2=1$, then we have concluded the proof, as this will show that $g$ was not extremal 
(recall we are assuming that $h\notin {\rm span}(g)$) and hence a contradiction.

We prove now the claim.  We compute
\begin{align*}
	|\DIFF^2(g+\epsilon h)|(\Omega)&=\sum_{k,\ell}| (a^g_k+\epsilon a_k^h) - (a_\ell^g+\epsilon a_\ell^h)|\HH^1(\partial P_k\cap \partial P_\ell)\\&
	=\sum_{k,\ell} | (a^g_k-a_\ell^g)+\epsilon (a_k^h- a_\ell^h)|\HH^1(\partial P_k\cap \partial P_\ell)
\end{align*}
and similarly we compute $|\DIFF^2 (g-\epsilon h)|$. 
Notice now that for every $k,\ell$ such that $\HH^1(\partial P_k\cap\partial P_\ell)>0$, then 
$a_k^h-a_\ell^h=\lambda_{k,\ell} (a_k^g-a_\ell^g)$ for some $\lambda_{k,\ell}$ such that $\epsilon|\lambda_{k,\ell}|\le 1$
(notice that $a_k^g=a_\ell^g$ implies $a_k^h=a_\ell^h$).  Therefore,
\begin{align*}
	&|\DIFF^2(g+\epsilon h)|(\Omega)+|\DIFF^2(g-\epsilon h)|(\Omega)\\&\qquad=\sum_{k,\ell}\big(| (a^g_k-a_\ell^g)+\epsilon (a_k^h- a_\ell^h)|
	+| (a^g_k-a_\ell^g)-\epsilon (a_k^h- a_\ell^h)|\big)\HH^1(\partial P_k\cap \partial P_\ell)\\&\qquad=\sum_{k,\ell}\big(
	|a_k^g-a_\ell^g| (1+\epsilon\lambda_{k,\ell})+
	|a_k^g-a_\ell^g| (1-\epsilon\lambda_{k,\ell})\big) \HH^1(\partial P_k\cap \partial P_\ell)\\&\qquad=2\sum_{k,\ell}
	|a_k^g-a_\ell^g|  \HH^1(\partial P_k\cap \partial P_\ell)=2|\DIFF^2 g|(\Omega)=2,
\end{align*}
which concludes the proof.
\end{proof}
\begin{prop}\label{dsacanos}
It holds that $${\rm {co}(\EE)}=\pi(\rm{CPWL}(\Omega))\cap\BB.$$
\end{prop}
\begin{proof}
Being the inclusion $\subseteq$ trivial by convexity, we focus on the opposite inclusion. We will heavily rely on Remark \ref{cpwlrem}. Take $g\in\pi(\rm{CPWL}(\Omega))\cap\BB$, $g\ne 0$. Now consider the set
\begin{equation}\notag
T\defeq\left\{h\in\mathcal E\cap\SSS:\supp(|\DIFF^2 h|)\subseteq\supp(|\DIFF^2 g|)\right\},
\end{equation}
and notice that by  Proposition \ref{fivedotfive}  and
 the fact that $g\in\rm CPWL(\Omega)$, then $T$ is finite (we will show that $T\ne\emptyset$ in \textbf{Step 1}). Also notice that $h\in T$ if and only if $-h\in T$, so that we write $T=\{\pm t_1,\dots,\pm t_\ell\}$.
We aim at showing that $g\in {\rm co}(T)$, this will  conclude the proof.  
\medskip\\\textbf{Step 1.} We claim that $T\ne \emptyset$. First, if $g\in\RR\EE$, then the whole proof is concluded, as $g/|\DIFF^2 g|(\Omega)\in T$ so that $g\in {\rm co}(T)$. Otherwise,   thanks to Proposition~\ref{fivedotfive}, we can take $h_1\in\BB$ with $\supp(|\DIFF^2 h_1|)\subseteq \supp(|\DIFF^2 g|)$ but $h_1\notin{\rm span}(g)$. Notice that this forces $h_1\in\pi({\rm CPWL}(\Omega))$. We can then take $\lambda_1\in\RR$ such that $$0< \HH^1(\supp(|\DIFF^2(g-\lambda_1 h_1)|))\le \HH^1(\supp(|\DIFF^2 g|))-\Lambda,$$
where $$\Lambda:=\min \{\HH^1(\partial P_k\cap \partial P_\ell): \HH^1(\partial P_k\cap \partial P_\ell)>0,\,\,
  k\neq\ell \}$$
and we are using the same notation as for Proposition \ref{fivedotfive} (here the  finitely many triangles are relative to $g$). 
If $g-\lambda_1 h_1\in\RR\EE$ then we have concluded the proof of this step. Otherwise, take $h_2\in\BB$ with $\supp(|\DIFF^2 h_2|)
\subseteq \supp(|\DIFF^2 (g-\lambda_1 h_1)|)$ but $h_2\notin{\rm span}(g-\lambda_1 h_1)$. Take then $\lambda_2\in\RR$ such that 
$$0< \HH^1(\supp(|\DIFF^2(g-\lambda_1 h_1-\lambda_2 h_2)|))\le \HH^1(\supp(|\DIFF^2 (g-\lambda_1 h_1)|))-\Lambda\le 
\HH^1(\supp(|\DIFF^2 g|))-2\Lambda.$$
If $g-\lambda_1 h_1-\lambda_2 h_2\in\RR\EE$, then the proof of this step is concluded. Otherwise we continue in this way, but, by the uniform decay posed on Hessian--Schatten total variations, this process must stop, and this forces eventually $g-\lambda_1 h_1-\lambda_2 h_2-\ldots- \lambda_s h_s\in\RR\EE$.
\medskip\\\textbf{Step 2.} We claim that $g\in{\rm span}(T)$. The proof of this fact is identical to the one of \textbf{Step 1}, but we take 
 $h_i\in T$
 instead of $h_i\in\BB$. The possibility of doing so is ensured by \textbf{Step 1} (applied to $g,g-\lambda_1 h_1,\ldots$) and process would stop when $g-\lambda_1 g_1-\lambda_2 h_2- \ldots- \lambda_s h_s=0$.
\medskip\\\textbf{Step 3.} We consider the finite dimensional vector subspace $\mathcal V\defeq{\rm span}(T)\subseteq {L^1(\Omega)}/{\mathcal{A}}$, endowed with the subspace topology. Consider also $\BB\cap \mathcal{V}$, compact and convex, notice that $g\in \BB\cap \mathcal{V}$, by \textbf{Step 2}. We claim that ${\rm ext}(\BB\cap \mathcal{V})\subseteq T$. This will conclude the proof by the Krein--Milman Theorem, as in \eqref{KM}, with $T$ in place of $\MM$ and $\BB\cap \mathcal{V}$ in place of $\BB$. We are using that $T$ is closed and that ${\rm co}(T)=\overline{\rm co}(T)$ as $T$ is finite. Take $h\in {\rm ext}(\BB\cap\mathcal{V})$, write then $h=\lambda_1 t_1+\ldots +\lambda_\ell t_\ell$. Then there exists $j\in \{1,\dots,\ell\}$ such that $\supp(|\DIFF^2 t_j|)\subseteq \supp(|\DIFF^2 h|)$, as $\supp(|\DIFF^2 h|)\subseteq\supp(|\DIFF^2 g|)$ and by \textbf{Step 1} applied to $h$ instead of $g$. The same perturbation argument of Proposition \ref{fivedotfive} shows that, in order for $h$ to be extremal in $\BB\cap\mathcal V$, we must have $h=\pm t_j$, which concludes the proof.
\end{proof}

\begin{thm}[Density of $\rm CPWL$ extreme points]\label{thm:extreme}
If $n=2$, then ${\rm ext}(\BB)\subseteq\overline\EE$. In particular, the extreme points of 
$$\{f\in L^1_\rmloc(\Omega):|\DIFF^2 f|(\Omega)\le 1\}$$
are contained in $\pi^{-1}(\overline{\EE})$ (recall that the closure is taken with respect to the 
quotient topology of $\displaystyle {L^1(\Omega)}/{\mathcal A}$).
If Conjecture \ref{maincnj} holds, this is true for any number $n$ of space dimensions.
\end{thm}
\begin{proof} By Proposition \ref{dsacanos}, 
$${\rm {co}(\EE)}=\pi(\rm{CPWL}(\Omega))\cap\BB,$$
so that the density Theorem \ref{mainthm} gives
$$
\rm\overline{co}(\EE)=\overline{\pi(\rm{CPWL}(\Omega))\cap \BB}=\BB.
$$
Then the claim follows from the Krein--Milman Theorem as recalled in  \eqref{KM}.
\end{proof}

 \section*{Acknowledgments}
This work was supported in part by the European Research Council (ERC Project FunLearn) under Grant 101020573 and in part by the PRIN MIUR project
``Gradient flows, Optimal Transport and Metric Measure Structures''.  The authors would like to thank Sergio Conti for helpful discussions. 
\bibliographystyle{alpha}
\bibliography{Biblio}

\end{document}